\documentclass[a4page]{article}

\usepackage[margin=1.5in]{geometry}

\usepackage{graphicx}%
\usepackage{multirow}%
\usepackage{amsmath,amssymb,amsfonts}%
\usepackage{amsthm}%
\usepackage{mathrsfs}%
\usepackage[title]{appendix}%
\usepackage{xcolor}%
\usepackage{textcomp}%
\usepackage{manyfoot}%
\usepackage{booktabs}%
\usepackage{algorithm}%
\usepackage{algorithmicx}%
\usepackage{algpseudocode}%
\usepackage{listings}%
\usepackage{mathabx} % 
\usepackage{anyfontsize}
\RequirePackage[colorlinks,citecolor=blue,urlcolor=blue,linkcolor=blue]{hyperref}
\usepackage[authoryear]{natbib}

\numberwithin{equation}{section}

\theoremstyle{thmstyleone}%

\newtheorem{theorem}{Theorem}[section]
\newtheorem{proposition}[theorem]{Proposition}
\newtheorem{assum}[theorem]{Assumption}% 
\newtheorem{cor}[theorem]{Corollary}% 
\newtheorem{lemma}[theorem]{Lemma}% 
\newtheorem{example}[theorem]{Example}%
\newtheorem{remark}[theorem]{Remark}%
\newtheorem{definition}[theorem]{Definition}%
\newtheorem{notation}[theorem]{Notation}

\newcommand{\Lp}[1]{\ensuremath{\mathbb{L}}^{#1}} 
\newcommand{\HS}[1]{\ensuremath{\left\Vert#1\right\Vert_{\mathrm{HS}}}} 
\newcommand{\HSgg}[1]{\ensuremath{\bigg\Vert#1\bigg\Vert_{\mathrm{HS}}}} 
\newcommand{\HSt}[1]{\ensuremath{\Vert#1\Vert_{\mathrm{HS}}}} 
 
\newcommand{\HSsct}[2]{\ensuremath{\langle#1,#2\rangle_{\text{HS}}}} 
\newcommand{\inte}{\thickspace\mathrm{d}} 
\newcommand{\indikator}{1\hspace{-0,9ex}1}
\newcommand{\Proj}[1]{\ensuremath{\mathrm{#1}}} 
\newcommand{\Projr}[1]{\ensuremath{\mathrm{P}_{>r}\left.\left(#1\right)\right|_{V_{>r}}}} 
\newcommand{\Projrgg}[1]{\ensuremath{\mathrm{P}_{>r}\bigg(#1\bigg)\bigg|_{V_{>r}}}}
\newcommand{\Projrt}[1]{\ensuremath{\mathrm{P}_{>r}(#1)|_{V_{>r}}}} 
\newcommand{\Projro}[1]{\ensuremath{\mathrm{P}_{>r}\left.#1\right|_{V_{>r}}}} 
\newcommand{\Projrogg}[1]{\ensuremath{\mathrm{P}_{>r}#1\bigg|_{V_{>r}}}} 
\newcommand{\tr}{\ensuremath{\mathrm{tr}}} 
\newcommand{\rank}{\ensuremath{\mathrm{rank}}} 
\newcommand{\avg}{\ensuremath{\mathrm{avg}}}
\newcommand{\spd}{\ensuremath{\mathrm{spd}}} 
\newcommand{\EW}[1]{\ensuremath{\mathbb{E}\left[#1\right]}} 
\newcommand{\EWest}[1]{\ensuremath{\mathbb{E}\big[#1\big]}}
\newcommand{\EWt}[1]{\ensuremath{\mathbb{E}[#1]}} 
\newcommand{\EWind}[2]{\ensuremath{\mathbb{E}_{#1}\left[#2\right]}} 
\newcommand{\EWcond}[2]{\ensuremath{\mathbb{E}\left[\left.#1 \right| #2\right]}} 
\newcommand{\EWcondt}[2]{\ensuremath{\mathbb{E}[#1 | #2 ]}} 
\newcommand{\Vari}[1]{\ensuremath{\mathbb{V}\hspace{-0.5pt}\mathrm{ar}\left(#1\right)}} 
\newcommand{\VariCond}[2]{\ensuremath{\mathbb{V}\hspace{-0.5pt}\mathrm{ar}\left(\left.#1\right|#2\right)}}
 
\newcommand{\PPind}[2]{\mathbb{P}_{#1}\left(#2\right)}
\newcommand{\PPindt}[2]{\mathbb{P}_{#1}(#2)}

\providecommand{\keywords}[1]
{
  \small	
  \textbf{Keywords:} #1
}

\title{Testing the rank of the spot covariance matrix of a multidimensional It\^o semi-martingale}

\author{ Janine Steck\textsuperscript{a} \\
\small{\textsuperscript{a}Institut für Mathematik, Humboldt-Universit\"{a}t zu Berlin,} \\ 
\small{\href{mailto:janine.steck@hu-berlin.de}{janine.steck{@}hu-berlin.de}}
}

\date{}

\begin{document}

\maketitle

\begin{abstract}
    This work develops a statistical test for the maximal rank of the deterministic instantaneous (or spot) covariance matrix of a continuous-time $\mathbb{R}^d$-valued It\^o semi-martingale $X(t)$ using high-frequency observations with a particular focus on the impact of an adapted drift. We explicitly account for the presence of an adapted drift process, which, as our results demonstrate, cannot be neglected, by introducing a re-centred covariance estimator instead of relying solely on a second moment estimator. Building on this estimator, we test the null hypothesis that the rank of the spot covariance matrix is at most $r<d$ for all $t$ against local alternatives in which the $(r+1)$th eigenvalue is greater than some vanishing signal detection rate. Critical values are derived in a non-asymptotic framework and can be significantly affected by a potential drift. However, the power analysis establishes asymptotic consistency for separation rates, which depend on the H\"older regularity of both the drift and the spot covariance matrix, as well as on a potential spectral gap $\underline{\lambda}_r \geq 0$ under the null hypothesis. Simulation results indicate that the covariance-based test achieves higher power across a wider range of alternatives compared to classical second moment-based procedures.
\end{abstract}

\noindent \small{\href{https://mathscinet.ams.org/mathscinet/msc/msc2020.html}{\textbf{MSC2020 subject classifications}}\textbf{:} 62G10, 62M07, 62M20, 60H10.}\\
\keywords{
empirical covariance matrix, rank test, non-parametric testing, signal detection rate, high-frequency observations, eigenvalue perturbation}

\section{Introduction}\label{sec1}
In this work we study the problem of inferring the rank of time-varying covariance matrices in a continuous-time framework. In statistics, and in particular in financial econometrics, the rank of a covariance matrix is closely linked to the number of latent risk factors driving asset price dynamics $X$ (see \cite{JacodPodolskij2013}, \cite{AITSAHALIA2017} and the references therein). We consider a $d$-dimensional continuous-time process $(X(t), \thinspace t \in [0,1])$ satisfying 
\begin{align*}
    \inte X(t) = b(t) \inte t + \sigma(t) \inte W(t), \quad t \in [0,1], 
\end{align*}
(see \eqref{eq:ItoSemi} for more details), observed at high frequency and discrete, synchronous time points $X(i/n)$ for $i = 0, \ldots, n$. Our aim is to infer the maximal rank of its \emph{instantaneous} or \emph{spot covariance matrix}. Equivalently, this corresponds to determining the minimal number of independent Brownian motions required to model the dynamics of $X$ (see \cite{Jacod2008}). In practice, high-frequency observations are available for assets such as stocks and currencies (see for instance \cite{JacodPodolskij2013}). 

\noindent\textbf{Motivation and comparison to the literature.} In mathematical finance, asset prices such as currencies and commodities are often modelled as multidimensional semi-martingales.  
This modelling framework is further justified by \cite{DelbaenSchachermayer1994}, who show that, under the assumption of \emph{no free lunch with vanishing risk} for simple integrands, the price process must be a semi-martingale. 

\indent The significance of investigating a drift term becomes apparent by considering the concept of a \emph{drift burst}. It was introduced by \cite{ChristensenKim2022Tdbh} and is defined as a short time interval in which the price paths of financial assets are locally dominated by the drift component of a continuous-time It\^o semi-martingale. The authors show that drift bursts are integral to price dynamics across equities, fixed income, currencies, and commodities. \cite{FloraReno} investigate the drift burst test statistic as a tool for detecting flash crashes, which are abrupt and severe price changes occurring over extremely short time intervals (see e.g.\ \cite{Golub2012High}), by identifying short-term V-shaped price reversals. \cite{Kolokolov02012026} introduces a non-parametric test for detecting drift bursts in pure-jump processes and apply it to identify intraday flash crashes and gradual jumps in cryptocurrency prices recorded at a high frequency. More generally, empirical evidence suggests that structural characteristics of It\^o semi-martingales often evolve gradually rater than abruptly. For instance \cite{HoffmannMichael2018Niog} show that gradual changes
in the jump characteristic of a discretely observed It\^o semi-martingale occur more often than abruptly ones. This provides additional motivation for considering smoothness assumptions on the drift component. 

\indent Moreover, \cite{Laurent_Shi_2022} state that a non-zero drift in continuous time need not result in a non-zero drift in discrete time, which is particularly relevant given that empirical data are observed discretely. While the drift term is often ignored in the high-frequency literature because diffusion dominates it asymptotically (see e.g.\ \cite{Andersen}; \cite{Barndorff-NielsenOleE.2002Eqvu}), its impact on estimators has been analysed in second-order asymptotics by \cite{BollerslevTim2020RS}.

\indent The preceding results in the literature indicate that the drift component can nevertheless play a non-trivial role in price dynamics and empirical modelling (cf.\ \cite{ChristensenKim2022Tdbh}; \cite{LAURENT2026}). Furthermore, \cite{LAURENT2026} draw on substantial empirical evidence to show that asset prices exhibit non-negligible drifts in both low-frequency (daily, weekly, monthly) and high-frequency settings. Moreover, the drift term can significantly affect volatility estimation accuracy as is confirmed by finite-sample theory (cf.\ \cite{Laurent2020Veajd}). 

\indent Our test setting is closest to that of \cite{ReissWinkelmann2023}, as we retain the high-frequency framework and signal-detection perspective while extending a Brownian martingale model to an It\^o semi-martingale model and adapt the test statistic to the re-centred estimator accordingly. This extension is non-trivial, as the fast convergence rates in rank detection, based on local averages over intervals of length $\mathcal{O}(n^{-1})$, imply that the drift component present in the Itô semi-martingale setting may contribute at the first-order asymptotic level. In addition, it also bears similarities to \cite{JacodPodolskij2013}, who study rank estimation and testing problems in joint semi-martingale models for $X(t)$ and $\sigma(t)$. However, their approach targets integer-valued rank and does not yield a clear concept of convergence rates or local alternatives as in our signal detection framework. Another connection to rank is provided by \cite{AITSAHALIA2017}, who show that the covariance structure of high-frequency asset return data exhibits a low-rank plus sparsity structure.   

\indent In addition to incorporating drift, another natural extension of the classical martingale setting would be to allow for noisy observations of the underlying process, as considered by \cite{reiss2026rank}. Despite the difference in the data structure, their testing problem is closely related to ours: they test the same null hypothesis that the rank of the spot covariance matrix is at most $r < d$ and provide non-asymptotic and asymptotic critical values. Furthermore, they achieve asymptotic consistency for a separation rate. 

\noindent\textbf{Model.} Motivated by these considerations, our work builds on and generalises the approach of them, who proposed a statistical methodology for testing the maximal rank of the spot covariance matrix $\Sigma(t)$ based on high-frequency observations under the assumption that $X(t)$ is a Brownian martingale. As a first step towards a more general continuous It\^o semi-martingale framework, we extend this setting by incorporating an adapted drift process while maintaining deterministic volatility. This leads to an It\^o semi-martingale model of the form
\begin{align} \label{eq:ItoSemi}
    \inte X(t) = b(t) \inte t + \sigma(t) \inte W(t), \quad t \in [0,1], 
\end{align}
defined on a filtered probability space $(\Omega, \mathcal{F}, (\mathcal{F}_t)_{t \geq 0}, \mathbb{P})$ with an adapted drift process $b(t) \in \mathbb{R}^d$ for all $ t \in [0,1]$ satisfying $\EWt{\int_{0}^{1}\Vert b(t) \Vert \inte t} < \infty$ and the fourth moment H\"older continuity condition, an $m$-dimensional standard Brownian motion $W(t)$, $t \in [0,1]$, and a deterministic spot volatility matrix $\sigma(t) \in \mathbb{R}^{d \times m}$ for all $t \in [0,1]$ satisfying $\int_{0}^{1} \tr(\sigma(t) \sigma(t)^{\top}) \inte t < \infty$. Define the \emph{spot covariance matrix} 
$$
    \Sigma(t) := \sigma(t)\sigma(t)^\top \in \mathbb{R}^{d \times d}, \quad t \in [0,1]. 
$$ 

\noindent\textbf{Estimator construction.} In this paper, we construct an estimator based on local averages of block-wise non-parametric covariance estimates, using discrete and synchronous high-frequency observations of the process \eqref{eq:ItoSemi}. \cite{ReissWinkelmann2023} assume $b \equiv 0$ and introduce the block-wise non-parametric \emph{second moment estimator} on blocks $I_k = [kh, (k+1)h]$, $k = 0,\ldots, h^{-1}-1$, for a small block length $h \in (0, 1]$ (with $nh, h^{-1} \in \mathbb{N}$) 
\begin{align} \label{eq:secmomest}
    \hat{\Sigma}^{kh}_{2\mathrm{nd}} := h^{-1} \sum_{i=1}^{nh} \left(\Delta X_{ik}\right)^{\otimes 2},
\end{align}
where $\Delta X_{ik} := X(kh + i/n) - X(kh + (i-1)/n)$ for $i = 1, \ldots,nh$, $k = 0, \ldots, h^{-1}-1$, and $v^{\otimes 2} := vv^\top$ for $v \in \mathbb{R}^{d}$. It is then shown that this estimates the block-averaged covariance matrix
\begin{align} \label{eq:avcov}
    \Sigma^{kh}_{\avg} := h^{-1} \int_{I_k} \Sigma(t) \inte t
\end{align}
with an error of order $\mathcal{O}_\mathbb{P}((nh)^{-1/2})$ if we assume $\Delta X_{ik} \sim \mathcal{N}(0,\int_{kh+(i-1)/n}^{kh+i/n} \Sigma(t) \inte t)$, where $i=1,\ldots, nh$ and $k = 0, \ldots, h^{-1}-1$. To explicitly account for the local covariance structure of $\Sigma$, we propose a block-wise \emph{re-centred estimator} 
\begin{align} \label{eq:covest}
    \hat{\Sigma}^{kh} := h^{-1} \sum_{i=1}^{nh} \bigg(\Delta X_{ik} - (nh)^{-1} \sum_{j=1}^{nh} \Delta X_{jk}\bigg)^{\otimes 2},
\end{align}
which mitigates drift-induced bias. 

\noindent\textbf{Testing problem and statistical insight.} Averaging spot covariance  matrices over blocks $I_k$, $k=0,\ldots,h^{-1}-1$, introduces a subtle challenge: the block-averaged covariance $\Sigma_{\avg}^{kh}$ defined in \eqref{eq:avcov} can have full rank even if $\rank(\Sigma(t)) \leq r < d$ holds for all $t \in [0,1]$. This is due to averaging spot covariance matrices with \emph{time-varying eigenspaces}. \cite{Jacod2008} emphasise that $\Sigma^{kh}_{\avg}$ provides no direct insight into the instantaneous rank of $\Sigma(t).$  However, while the $(r+1)$th eigenvalue $\lambda_{r+1}(\Sigma(t))$ may be non-zero it remains small for sufficiently small block size $h \in (0,1]$ if $\Sigma(t)$ varies smoothly in time. Our perturbation bounds show that the maximal size of $\lambda_{r+1}(\Sigma^{kh}_{\avg})$ depends not only on the H\"older regularity parameter $\beta_{\Sigma} \in (0,1]$ of $\Sigma$, but also critically on the existence of a \emph{spectral gap} $\underline{\lambda}_{r} := \inf_{t \in [0,1]} \lambda_r(\Sigma(t))$, noting that $\lambda_{r+1}(\Sigma(t)) = 0$ for all $t \in [0,1]$ under the null hypothesis, as well as on the H\"older regularity parameter $\beta_b \in (0,1]$ of the drift process $b$. Small eigenvalues $\lambda_{r+1}(\hat{\Sigma}^{kh})$ across the blocks $I_k$ should favour non-rejection of the null hypothesis, whereas large eigenvalues should lead to its rejection. Our modified test exploits this property, while paying particular attention to the choice of block sizes $h \in (0,1]$, non-asymptotic critical values and signal detection rates. 

\indent Similar to \cite{ReissWinkelmann2023, reiss2026rank}, we do not test the rank directly but construct a test statistic based on the $(r+1)$th largest eigenvalue aggregated over small blocks of size $h \in (0, 1]$. However, our statistic is adapted to the re-centred estimator, which allows us to control the effect of potential drift terms while preserving the core idea of the original test. A theoretical contribution is that the behaviour of these eigenvalues depends not only on the H\"older regularity of the covariance matrix $\Sigma$ but also on the H\"older regularity of the drift process $b$. This interplay requires a refined perturbation analysis. 

\noindent\textbf{Main contributions.} We extend the rank testing framework to It\^o semi-martingales and introduce H\"older regularity assumptions on the drift process $b$. Our analysis cannot exploit the convenient distributional properties that arise when the estimator is a convex combination of independent Wishart matrices, as the re-centring step introduces dependence. It is therefore based on a systematic decomposition of the re-centred estimator, combined with the use of conditional expectations to control its stochastic behaviour. In particular, the assumption of H\"older-regularity of the drift $b$ plays a crucial role in this control. 

\indent Under the null hypothesis, the leading convergence rates coincide with those established in Theorem~3.3 of \cite{ReissWinkelmann2023}, due to the refined control of the re-centred estimator. These rates depend not only on the regularity of the covariance matrix function $\Sigma$ and the existence of a spectral gap $\underline{\lambda}_r$, but also on the H\"older regularity of the drift process $b$, which distinguishes our setting from theirs. However, under local alternatives, the signal detection rates deteriorate compared to those obtained in Corollary 3.7 of the same reference. Moreover, in the special case of a constant drift, our method requires a larger number of observations, namely of order $\min\{n^{-1/5}, \underline{\lambda}_{r}^{-2/3} n^{-1/3}\}$ than in their approach, where even block sizes of order $n^{-1}$ are admissible. This highlights a trade-off between improved control of the test statistic reflected in smaller (numerical) critical values (cf.\ Table~\ref{tab:crit}) and statistical efficiency, as our method requires a larger sample size. 
\begin{figure}[t]
    \centering
    \includegraphics[width=\linewidth]{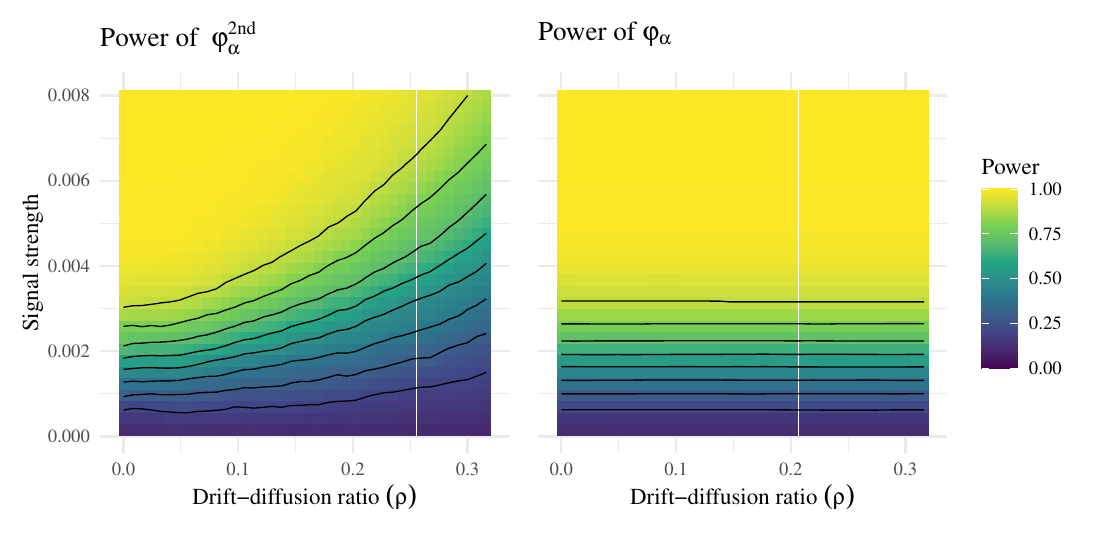}
    \caption{Power of the rank test as a function of signal strength and drift-diffusion ratio $\rho$ (see \eqref{eq:rho}) in Example~\ref{ex:3.10MLVaryingLinDrift} with $\underline{\lambda}_1 = 0.5$ and $\beta_{\Sigma} = 0.5$. \emph{Left:} Rank test at level $\alpha = 0.1$ based on the second moment estimator defined in \eqref{eq:testsec}. \emph{Right:} Rank test at level $\alpha = 0.1$ based on the covariance estimator defined in \eqref{eq:Test}. See Section \ref{sec_nu} for the parameters chosen in this simulation.}
    \label{fig:intro}
\end{figure}
Overall, the covariance-based test appears substantially more robust to increasing drift contributions in the Itô semi-martingale, maintaining a nearly unchanged power, whereas the power of the second moment-based test deteriorates markedly (see Fig \ref{fig:intro}).

\noindent\textbf{Organisation of the paper.} In this section, we continue by introducing some notation. We then present the testing problem in Section~\ref{sec:test}. The null hypothesis is formulated in Subsection~\ref{sec:h0}, while the corresponding general results on its behaviour under the null hypothesis are established in Subsection~\ref{sec:mainh0}. Results for the population version of the test statistics are stated in Subsection~\ref{sec:poph0}.
In Subsection~\ref{sec:alternative}, local alternatives are considered. Finally, Section~\ref{sec_nu} visualises the results obtained in the previous section compared to the performance of the second moment estimator. More precisely, our simulation results indicate that, in the presence of a non-zero drift, the second moment-based test exhibits a substantial loss of power relative to our covariance-based test. Proofs of the statements of Subsections~\ref{sec:mainh0} and \ref{sec:alternative} are given in \hyperref[app:proofs]{Appendix}. 

\paragraph*{Notation.} 
For all $a, b \in \mathbb{R}$, write $a \vee  b := \max\{ a, b\}$ and $a \wedge  b := \min\{ a, b\}$. Further, we define $\ldbrack a, b \rdbrack := [a,b] \cap \mathbb{N}_{0}$ for $a, \, b \in \mathbb{R}_{\geq 0}$ with $a < b$, where $\mathbb{N}_{0} := \mathbb{N} \cup \{0\}$ and $\mathbb{R}_{\geq 0} := [0, \infty)$. 

\indent Let $(a_n)_{n \geq 1}$, $(b_n)_{n \geq 1}$ be two sequences. We write $a_n \lesssim b_n$ if there exists a constant $C>0$, independent of $n$ (but possibly depending on the dimension $d$) such that $a_n \leq C b_n$ for all $n \in \mathbb{N}$ (or other parameters for the asymptotics). The notation $a_n \sim b_n$ means that $a_n \lesssim b_n$ and $b_n \lesssim a_n$. Throughout, constants may depend on the fixed dimension $d$ without explicit mention.

\indent In what follows, $\indikator_d$ denotes the identity matrix in $\mathbb{R}^{d \times d}$. We set $v^{\otimes 2} = vv^{\top} \in \mathbb{R}^{d \times d}$ for any vector $v \in \mathbb{R}^{d}$. We denote by $v \odot w$ the symmetric sum $vw^{\top} + wv^{\top}$ for any vectors $v,w \in \mathbb{R}^{d}$. We write $\mathbb{R}^{d \times d}_{\mathrm{sym}}$ and $\mathbb{R}^{d \times d}_{\spd}$ for $\{S \in \mathbb{R}^{d \times d}\, | \, S = S^\top\}$ and $\{S \in \mathbb{R}^{d \times d}_{\mathrm{sym}} \, | \, S\text{ positive semi-definite}\}$, respectively. For $A, \, B \in \mathbb{R}^{d \times d}_{\mathrm{sym}}$ the partial order $A \preccurlyeq B$ says that $B-A \in \mathbb{R}^{d \times d}_{\spd}$. Given $A \in \mathbb{R}^{d \times d}_{\mathrm{sym}}$, we denote the ordered eigenvalues by $\lambda_{\mathrm{max}}(A) := \lambda_1(A) \geq \ldots \geq \lambda_d(A) =: \lambda_{\mathrm{min}}(A)$ (according to their multiplicities), $\tr(A) := \sum_{i=1}^{d} \lambda_{i}(A)$ and the remaining $\tr_{>r}(A) := \sum_{i=r+1}^{d} \lambda_{i}(A)$. The Hilbert--Schmidt (or Frobenius) inner product between matrices $A, \, B \in \mathbb{R}^{d \times d}$ is defined as $\HSsct{A}{B} = \tr(B^{\top}A)$ with Hilbert--Schmidt norm $\HSt{A} = \HSsct{A}{A}^{1/2}$, and the spectral norm is given by $\Vert A \Vert = \max_{v \in \mathbb{R}^d, \, \Vert v \Vert = 1} \Vert Av \Vert$, where $A \in \mathbb{R}^{d \times d}$. 

\indent By $C(U;V)$ we denote the continuous functions from $U$ into $V$, where $U,\ V$ are metric spaces. Let us recall the definition of matrix-valued $\beta_{\Sigma}$-H\"older balls of radius $L_{\Sigma} > 0$ with respect to the spectral norm. For $\beta_{\Sigma} \in (0,1]$, we define
$$
    \mathcal{C}^{\beta_{\Sigma}}\left(L_{\Sigma}\right) := \left\{\left.\Sigma:[0,1] \rightarrow \mathbb{R}^{d \times d}_{\spd} \, \right| \, \Vert \Sigma(t) - \Sigma(s) \Vert \leq L_{\Sigma}|t - s|^{\beta_{\Sigma}} \text{ for all } t,\, s \in [0,1] \right\}.
$$

\section{The testing problem}\label{sec:test}
\indent We briefly recall the underlying process, the estimator of interest in our setting and the realised covariance matrix. To this end, consider the $d$-dimensional Itô semi-martingale model
\begin{align*} 
    \inte X(t) = b(t) \inte t + \sigma(t) \inte W(t), \quad t \in [0,1], 
\end{align*}
defined on a filtered probability space $(\Omega, \mathcal{F}, (\mathcal{F}_t)_{t \geq 0}, \mathbb{P})$, where $b(t) \in \mathbb{R}^d$ is an adapted drift process satisfying
$\mathbb{E}[\int_0^1 \|b(t)\| \inte t] < \infty$ and the fourth moment H\"older continuity condition, $W(t)$, $t \in [0,1]$, is an $m$-dimensional standard Brownian motion, and $\sigma(t) \in \mathbb{R}^{d \times m}$ is a deterministic spot volatility matrix satisfying
$\int_0^1 \tr(\sigma(t)\sigma(t)^\top)\inte t < \infty.$ 

\indent To construct our test, we partition the interval $[0,1]$ into $h^{-1}$ blocks of equal length $h \in (0,1]$, assuming that both $nh$ and $h^{-1}$ are integers. Each block is defined as $I_k := [kh, (k+1)h]$ for $k \in \ldbrack 0, h^{-1} - 1 \rdbrack$. 

\indent The re-centred estimator is defined by 
$$
    \hat{\Sigma}^{kh} := h^{-1} \sum_{i=1}^{nh}\bigg(\Delta X_{ik} - (nh)^{-1} \sum_{j=1}^{nh} \Delta X_{jk} \bigg)^{\otimes 2},
$$
where $\Delta X_{ik} := X(kh + i/n) - X(kh + (i-1)/n)$ for $i \in \ldbrack 1, nh \rdbrack$, $k \in \ldbrack 0, h^{-1}-1 \rdbrack$ and the realised covariance matrix is defined by 
$$
    \Sigma_{\avg}^{kh} = h^{-1} \int_{I_k} \Sigma(t) \inte t, 
$$
where $k \in \ldbrack 0, h^{-1}-1\rdbrack$. Both quantities are defined and analysed block-wise.  

\indent Our objective is to test the null hypothesis that the rank of the time-varying covariance matrix $\Sigma(t)$ does not exceed some arbitrary but fixed rank $r \in \ldbrack 0, d-1 \rdbrack$ at any time point $t \in [0,1]$, i.e.\
$$
    \mathcal{H}_0: \thickspace \max_{t \in [0,1]} \rank\left(\Sigma(t)\right) \leq r.
$$
To this end, we introduce the statistic
\begin{align}\label{eq:Statistic}
    T_{n,h}:=\sum_{k=0}^{h^{-1}-1} h \lambda_{r+1}\big(\hat{\Sigma}^{kh}\big),
\end{align}
which aggregates the $(r+1)$th largest eigenvalues of the estimated covariance matrix across all blocks. A level-$\alpha$ test $\varphi_\alpha$ with $\alpha \in (0,1)$ is then defined as
\begin{align}\label{eq:Test}
    \varphi_{\alpha} :=  \indikator\left(T_{n,h} > \kappa_{\alpha}\right),
\end{align}
where $\kappa_{\alpha}$ denotes the critical value corresponding to the desired significance level $\alpha$. Each block must contain at least $r+1$ increments; otherwise, $\lambda_{r+1}(\hat{\Sigma}^{kh}) = 0$ by definition. 

\subsection{The null hypothesis}\label{sec:h0}

We are interested in testing whether the covariance function $\Sigma : [0,1] \rightarrow \mathbb{R}^{d \times d}$ has rank at most $r$ on the interval $[0,1]$, where $r \in \ldbrack0,d-1\rdbrack$ is fixed. A natural test statistic is obtained by aggregating, over all $h^{-1}$ blocks, the $(r+1)$th largest eigenvalues of the local covariance estimator $\hat{\Sigma}^{kh}$, defined in \eqref{eq:Statistic}. 

\indent Before formally defining the null hypothesis, we first introduce a notion of Hölder regularity for the drift process $b$.

\begin{assum}\label{ass:hölderad}
    Let $b: [0,1] \rightarrow \mathbb{R}^{d}$ be an adapted drift process such that for all $0 \leq t \leq s \leq 1$, there exists a constant $L_b \geq 0$ and an exponent $\beta_b \in (0,1]$ such that
    \begin{align} \label{eq:HölderAd}
        \EW{\Vert b(s) - b(t) \Vert^4}^{1/4} \leq L_b |s-t|^{\beta_b}.
    \end{align}
    We denote by $\mathcal{H}^{\beta_b}(L_b)$ the class of all adapted $\mathbb{R}^{d}$-valued processes $b$ such that $b$ satisfies the H\"older-regularity condition \eqref{eq:HölderAd} and $\EWt{\int_{0}^{1} \Vert b(t) \Vert \inte t} < \infty$ .
\end{assum}   

\indent We now formally state the null hypothesis to be tested.

\begin{definition}\label{defi:null}
    Throughout, we consider the null hypothesis that the spot covariance matrix $\Sigma(t)$ has at most rank $r\in\ldbrack 0, d-1\rdbrack$ for all $t \in [0,1]$. Moreover, we assume H\"older-regularity conditions on $b$ and $\Sigma$, with exponents $\beta_b, \, \beta_{\Sigma} \in (0,1]$, constants $L_b \geq 0, \ L_{\Sigma} > 0$, and, potentially, a spectral gap $\underline{\lambda}_r>0$. We set 
    \begin{align*}
        \mathcal{H}_0 
        &:= \mathcal{H}_0(r, \beta_b, \beta_{\Sigma}, L_b, L_{\Sigma}) \\
        &:= \left\{(\Sigma, b) \in C \left([0,1]; \mathbb{R}^{d \times d}_{\spd}\right) \times \mathcal{H}^{\beta_b}(L_b) \, \left| \, \sup_{t\in[0,1]} \mathrm{rank}(\Sigma(t)) \leq r, \, \Sigma \in \mathcal{C}^{\beta_{\Sigma}}\left(L_{\Sigma}\right) \right. \right\}, \\
        \mathcal{H}_0^{\mathrm{gap}} &:= \mathcal{H}_0^{\mathrm{gap}}(r, \beta_{b}, \beta_{\Sigma}, L_b, L_{\Sigma}, \underline{\lambda}_r) \\
        &:= \left\{(\Sigma, b) \in C \left([0,1]; \mathbb{R}^{d \times d}_{\spd}\right) \times \mathcal{H}^{\beta_b}(L_b) \, \left| \, \sup_{t\in[0,1]} \mathrm{rank}(\Sigma(t)) \leq r, \, \Sigma \in \mathcal{C}^{\beta_{\Sigma}}\left(L_{\Sigma}\right), \right. \right.\\
        &\hspace{8cm}\qquad \left.\inf_{t \in [0,1]} \lambda_r(\Sigma(t)) \geq \underline{\lambda}_r \right\}.
     \end{align*}
\end{definition}

\newpage
\begin{remark} \leavevmode
    \begin{enumerate}
        \item[(a)] In addition to the H\"older-regularity condition on $\Sigma$ with parameters $\beta_{\Sigma}$ and $L_{\Sigma}$, we impose under the null hypothesis a spectral gap condition on the $r$th eigenvalue of the covariance matrix $\Sigma$, quantified by a lower bound $\underline{\lambda}_r > 0$. For completeness, we also allow for the degenerate case $\underline{\lambda}_r = 0$, in which no such constraint is imposed. 
        \item[(b)] As we also incorporate a drift term, minimal regularity assumptions on $b$ are natural; see Definition~\ref{defi:null}. These assumptions will play a role when deriving the non‑asymptotic critical values for our test \eqref{eq:Test}. The drift term $b$ is included in the null hypothesis solely to ensure that the model is well defined; it remains a nuisance parameter.
        \item[(c)] The influence of the parameters becomes apparent through the inclusion 
    $$
        \mathcal{H}_0^{\mathrm{gap}}(r, \beta_b, \beta_{\Sigma}, L_b,  L_{\Sigma}, \underline{\lambda}_r) \subseteq \mathcal{H}_0^{\mathrm{gap}}(r, \beta_b^{\prime}, \beta_{\Sigma}^{\prime}, L_{b}^{\prime}, L_{\Sigma}^{\prime}, \underline{\lambda}_{r}^{\prime})
    $$ 
    whenever $\beta_b \geq \beta_b^{\prime}$, $\beta_{\Sigma} \geq \beta_{\Sigma}^{\prime}$, $L_b \leq L_b^{\prime}$, $L_{\Sigma} \leq L_{\Sigma}^{\prime}$ and $\underline{\lambda}_r \geq \underline{\lambda}_{r}^{\prime}$. 
    \end{enumerate}  
\end{remark}

Before we can determine the critical values $\kappa_{\alpha}$ for the level-$\alpha$ test $\varphi_{\alpha}$, we introduce an appropriate projection. It specifies the subspace onto which the reduction is applied. Let $k \in \ldbrack 0, h^{-1} -1 \rdbrack$ and $r \in \ldbrack 0, d-1 \rdbrack$. For each block $I_k$ we consider the corresponding eigenspace
$$
    V_{>r} := \operatorname{span}(v_{r+1}, \ldots, v_d),
$$
where $v_{r+1}, \ldots, v_d$ are the eigenvectors corresponding to the $(d-r)$ smallest eigenvalues
$$
    \lambda_{r+1}\left(\Sigma_{\avg}^{kh}\right) \geq \ldots \geq \lambda_d \left(\Sigma_{\avg}^{kh}\right) \geq 0
$$
of $\Sigma_{\avg}^{kh}$. Let $\Proj{P}_{>r}$ denote the orthogonal projection onto $V_{>r}$. For any matrix
$$
    S \in \big\{\hat{\Sigma}^{kh}, \Sigma_{\avg}^{kh}, \Sigma(t)\big\},
$$
we define the $(d-r)\times(d-r)$ lower minor as
$$
    S_{>r} := \Projro{S}.
$$
This construction allows us to restrict the analysis to the relevant subspace. By the Cauchy interlacing theorem (see e.g.\ \cite{Johnstone2001} or \cite{Tao2012}), we obtain
\begin{align}\label{eq:TUB}
    T_{n,h} = \sum_{k=0}^{h^{-1}-1} h\lambda_{r+1}\big(\hat{\Sigma}^{kh}\big) \leq \sum_{k=0}^{h^{-1}-1} h \lambda_{\max}\big(\hat{\Sigma}_{>r}^{kh}\big).
\end{align}
Applying the triangle inequality then yields
\begin{align}\label{Eq:EigenvalueProjection}
    \lambda_{\max}\big(\hat{\Sigma}_{>r}^{kh}\big) = \big\Vert \hat{\Sigma}_{>r}^{kh} \big\Vert\leq \big\Vert\EWest{\hat{\Sigma}_{>r}^{kh}}\big\Vert + \big\Vert\hat{\Sigma}_{>r}^{kh} - \EWest{\hat{\Sigma}_{>r}^{kh}}\big\Vert.
\end{align}

\subsection{General results under the null hypothesis} \label{sec:mainh0}
\indent In this section we calibrate the test under the null hypothesis. To this end, we impose an assumption on the relationship between the sample size $n$ and the block length $h$, ensuring that the block length does not decrease faster than the inverse sample size.
\begin{assum}\label{ass:nh}
    We assume that $n^{-1} \lesssim h$. 
\end{assum}
\begin{remark}
    The condition $h \geq n^{-1}$ is minimal, since intervals of length $h < n^{-1}$ do not contain any local observations on the sampling scale.
\end{remark}

By Assumption \ref{ass:nh}, we derive an upper bound for the bias, as stated in the following proposition.
 
\begin{proposition}[Upper bound for the bias] \label{prop:biash0ad}
    Grant Assumption~\ref{ass:nh}. Let $k \in \ldbrack 0, h^{-1}-1\rdbrack$ and let $(\Sigma, b) \in \mathcal{H}_0$ or $(\Sigma,b)\in\mathcal{H}_0^{\mathrm{gap}}$, depending on the existence of the spectral gap. Then we obtain  
    \begin{align*}
        \sum_{k=0}^{h^{-1}-1} h \big\Vert \EWest{\hat{\Sigma}_{>r}^{kh}}\big\Vert &\lesssim \bigg(\sum_{k=0}^{h^{-1}-1} h \lambda_{r+1}\big(\Sigma_{\avg}^{kh}\big)\bigg) + L_b^2 n^{-1}h^{2\beta_b}.
    \end{align*}
\end{proposition}
  
Having understood the expectation under the null hypothesis, we now turn to the deviations. The next result provides an upper bound for the stochastic error.

\begin{proposition}[Upper bound for the stochastic error] \label{prop:stocherrorh0ad}
    Let Assumption~\ref{ass:nh} hold. If $k \in \ldbrack 0, h^{-1}-1\rdbrack$ and if $(\Sigma, b) \in \mathcal{H}_0$ or $(\Sigma, b) \in \mathcal{H}_{0}^{\mathrm{gap}}$, depending on whether a spectral gap exists, then we have
    \begin{align*}
        \sum_{k=0}^{h^{-1}-1} h\EWest{\big\Vert \hat{\Sigma}_{>r}^{kh} - \EWest{\hat{\Sigma}_{>r}^{kh}} \big\Vert} \lesssim \bigg(\sum_{k=0}^{h^{-1}-1} h \lambda_{r+1}\big(\Sigma_{\avg}^{kh}\big)\bigg) + L_b^2 n^{-1}h^{2\beta_b}.
    \end{align*}
\end{proposition}
Before we state the proof ideas, we introduce the following notation. 
\begin{notation}\label{defi:form}
    Let $i \in \ldbrack 1, nh \rdbrack$ and $k \in \ldbrack 0, h^{-1}-1 \rdbrack$. For convenience, let $I_{ik} := [kh + (i-1)/n,\; kh + i/n]$ and $I_{k} := [kh, (k+1)h]$. We use the abbreviation $\mathcal{F}_{i-1, k} := \mathcal{F}_{kh + (i-1)/n}$.
    In the following, the object of interest is defined first, followed by its block-wise average. We have
    \begin{align*}
        B_{ik} &:= \int_{I_{ik}} b(t) \inte t, &\quad \overline{B_k} &:= (nh)^{-1} \sum_{j=1}^{nh} B_{jk}, \\ 
        B_{ik}^{\mathrm{cond}}&:=\int_{I_{ik}} \EWcond{b(t)}{\mathcal{F}_{i-1,k}} \inte t, &\quad \overline{B_{k}^{\mathrm{cond}}} &:= (nh)^{-1} \sum_{j=1}^{nh} B_{jk}^{\mathrm{cond}}, \\
        \Delta X_{ik} &:= B_{ik} + \int_{I_{ik}} \sigma(t) \inte W(t), &\quad \overline{\Delta X_{k}} &:= (nh)^{-1} \sum_{j=1}^{nh} \Delta X_{jk}, \\ 
        \Delta \tilde{X}_{ik} &:= \Delta X_{ik} - B_{ik}, &\quad \overline{\Delta \tilde{X}_k} &:= (nh)^{-1} \sum_{j=1}^{nh} \Delta \tilde{X}_{jk}, \\
        \Delta \check{X}_{ik} &:= \Delta X_{ik} - B^{\mathrm{cond}}_{ik}, &\quad \overline{\Delta \check{X}_k} &:= (nh)^{-1} \sum_{j=1}^{nh} \Delta \check{X}_{jk}.
    \end{align*}
\end{notation}
\begin{proof}[Proof of Proposition \ref{prop:biash0ad} based on an error decomposition.]
    Throughout the proof, let $k \in \ldbrack 0, h^{-1}-1 \rdbrack$. We start by decomposing $\EWt{\hat{\Sigma}^{kh}_{>r}}$ as 
    \begin{align}
         h \big\Vert\EWest{\hat{\Sigma}_{>r}^{kh}}\big\Vert 
         &\leq \left\Vert \EW{\mathcal{T}_{1, >r}} \right\Vert + \left\Vert \EW{\mathcal{T}_{2,>r}} \right\Vert + 2\left\Vert \EW{\mathcal{T}_{3,>r}}\right\Vert, \label{eq:biasdecdd}
    \end{align}
    where 
    \begin{align}
        \mathcal{T}_{1, >r} &:= \Projrgg{\sum_{i=1}^{nh}\left(\Delta \check{X}_{ik} - \overline{\Delta \check{X}_{k}}\right)^{\otimes 2}}, \label{eq:T1Proj}\\
        \mathcal{T}_{2,>r} &:= \Projrgg{\sum_{i=1}^{nh}\left(B_{ik}^{\mathrm{cond}} - \overline{B_{k}^{\mathrm{cond}}}\right)^{\otimes 2}}, \label{eq:T2Proj}\\
        \mathcal{T}_{3,>r} &:= \Projrgg{\sum_{i=1}^{nh} \left(B_{ik}^{\mathrm{cond}} - \overline{B_{k}^{\mathrm{cond}}}\right)\left(\Delta \check{X}_{ik} - \overline{\Delta \check{X}_{k}}\right)^\top}. \label{eq:T3Proj}
    \end{align}
    Using this decomposition, we get the following upper bounds.  
    \begin{align}
        \left\Vert \EW{\mathcal{T}_{1, >r}} \right\Vert &\lesssim h \lambda_{r+1}\big(\Sigma_{\avg}^{kh}\big) + L_b n^{-1/2}h^{\beta_b + 1}\big(\tr_{>r}\big(\Sigma_{\avg}^{kh}\big)\big)^{1/2} + L_b^2 n^{-1}h^{2\beta_b + 1}, \label{eq:EWT1Proj} \\
        \left\Vert \EW{\mathcal{T}_{2, >r}} \right\Vert &\lesssim L_b^2n^{-1}h^{2\beta_b+1}, \label{eq:EWT2Proj} \\
        \left\Vert \EW{\mathcal{T}_{3, >r}} \right\Vert &\lesssim L_b n^{-1/2}h^{\beta_b + 1} \big(\tr_{>r}\big(\Sigma_{\avg}^{kh}\big)\big)^{1/2} + L_b^2 n^{-1} h^{2\beta_b+1}.\label{eq:EWT3Proj}
    \end{align}
    The upper bounds are obtained by standard arguments involving several technical steps. Detailed derivations, together with the final combination of these upper bounds yielding the claim, are provided in Appendix~\ref{app:biash0ad}.
\end{proof}
\begin{proof}[Proof of Proposition \ref{prop:stocherrorh0ad} based on an error decomposition.]
For the re-centred estimator $\hat{\Sigma}^{kh}$, the squared deviation from its expectation admits the decomposition 
    \begin{align*}
        &\EWest{\big\Vert \hat{\Sigma}_{>r}^{kh} - \EWest{\hat{\Sigma}_{>r}^{kh}} \big\Vert^2}\\
        &\leq 6 \left( \EW{\left\Vert \mathcal{T}_{4,>r} \right\Vert^2 + \left\Vert \mathcal{T}_{5,>r}\right\Vert^2 + 2\left\Vert \mathcal{T}_{6, >r}\right\Vert^2} + n^{-2}h^{-4} \EW{\left\Vert \mathcal{T}_{7, >r} \right\Vert^2 + \left\Vert \mathcal{T}_{8, >r} \right\Vert^2 + \left\Vert \mathcal{T}_{9, >r} \right\Vert^2} \right),
    \end{align*}
    where 
    \begin{align}
        \mathcal{T}_{4, >r} &:= h^{-1}\sum_{i=1}^{nh} \bigg(\Projro{\big(\Delta \tilde{X}_{ik} \big)^{\otimes 2}} - \EW{\Projro{\big(\Delta \tilde{X}_{ik}\big)^{\otimes2}}}\bigg), \label{eq:T4Proj}\\
        \mathcal{T}_{5, >r} &:= h^{-1}\sum_{i=1}^{nh} \bigg(\Projrogg{\bigg(\int_{I_{ik}} \left(b(t) - b(kh)\right) \inte t\bigg)^{\otimes2}} \notag\\
        &\phantom{:=}\qquad\qquad\qquad\qquad\qquad\qquad\qquad- \mathbb{E}\bigg[\Projrogg{\bigg(\int_{I_{ik}} \left(b(t) - b(kh)\right) \inte t\bigg)^{\otimes2}}\bigg]\bigg), \label{eq:T5Proj}\\
        \mathcal{T}_{6,>r} &:= h^{-1}\sum_{i=1}^{nh} \bigg(\Projrgg{\big(\Delta \tilde{X}_{ik} \big)\bigg(\int_{I_{ik}} \left(b(t) - b(kh)\right) \inte t\bigg)^{\top}} \notag\\
        &\phantom{:=}\qquad\qquad\qquad\qquad\quad\,- \mathbb{E}\bigg[\Projrgg{\big(\Delta \tilde{X}_{ik} \big)\left(\int_{I_{ik}} \left(b(t) - b(kh)\right) \inte t\right)^{\top}}\bigg]\bigg), \label{eq:T6Proj}\\
        \mathcal{T}_{7, >r} &:= \Projro{\left(nh \overline{\Delta \tilde{X}_k}\right)^{\otimes 2}} - \mathbb{E}\bigg[\Projro{\left(nh \overline{\Delta \tilde{X}_k}\right)^{\otimes2}}\bigg], \label{eq:T7Proj}\\
        \mathcal{T}_{8, >r} &:= \Projrogg{\left(\int_{I_k} \left(b(t)-b(kh)\right) \inte t\right)^{\otimes2}} - \mathbb{E}\bigg[\Projrogg{\left(\int_{I_k} \left(b(t)-b(kh)\right) \inte t\right)^{\otimes2}}\bigg], \label{eq:T8Proj}\\
        \mathcal{T}_{9, >r} &:= \Projrgg{\left(nh \overline{\Delta \tilde{X}_{k}} \right) \bigg(\int_{I_k} \left(b(t)-b(kh) \right) \inte t\bigg)^{\top}} \notag\\ 
        &\phantom{:=}- \mathbb{E}\bigg[\Projrgg{\left(nh \overline{\Delta \tilde{X}_{k}} \right) \bigg(\int_{I_k} \left(b(t)-b(kh)\right) \inte t\bigg)^{\top}}\bigg] \label{eq:T9Proj}
    \end{align}
    (see Appendix~\ref{app:stocherroroh0add} for more details). Moreover, we obtain the following upper bounds for each term separately
    \begin{align}
        \EW{\left\Vert\mathcal{T}_{4,>r}\right\Vert^2} &\lesssim \big(\tr_{>r}\big(\Sigma_{\avg}^{kh}\big)\big)^2, \label{eq:EWT4Proj} \\
        \EW{\Vert\mathcal{T}_{5,>r}\Vert^2} &\leq L_b^4 n^{-2}h^{4 \beta_b}, \label{eq:EWT5Proj} \\
        \EW{\left\Vert\mathcal{T}_{6,>r}\right\Vert^2} &\lesssim L_b^2 n^{-1} h^{2 \beta_b} \tr_{>r}\big(\Sigma_{\avg}^{kh}\big), \label{eq:EWT6Proj} \\
        \EW{\left\Vert \mathcal{T}_{7,>r}\right\Vert^2} &\lesssim h^2\big(\tr_{>r}\big(\Sigma_{\avg}^{kh}\big)\big)^2 \label{eq:EWT7Proj}, \\
        \EW{\left\Vert \mathcal{T}_{8,>r} \right\Vert^2} &\leq L_b^4 h^{4(\beta_b + 1)}, \label{eq:EWT8Proj} \\
        \EW{\left\Vert\mathcal{T}_{9,>r}\right\Vert^2} &\lesssim L_b^2 h^{2\beta_b+ 3} \tr_{>r}\big(\Sigma_{\avg}^{kh}\big). \label{eq:EWT9Proj} 
    \end{align}
    The upper bounds are obtained by standard arguments involving several technical steps. Their detailed derivation and final combination to conclude the claim are therefore deferred to Appendix~\ref{app:stocherroroh0add}.
\end{proof}

Next, we establish several useful properties and recall a necessary mathematical prerequisite.

\begin{remark} Let $i \in \ldbrack 1, nh \rdbrack$ and $k \in \ldbrack 0, h^{-1}-1 \rdbrack$. 
    \begin{enumerate}
        \item The increments $\Delta \tilde{X}_{ik} \sim \mathcal{N}(0,\int_{I_{ik}}\Sigma(t) \inte t)$ coincide with the centred increments considered by \cite{ReissWinkelmann2023}. 
        \item Replacing $\Delta X_{ik}$ by $\Delta \tilde{X}_{ik}$ in the second moment estimator \eqref{eq:secmomest} we obtain 
        $$
            \mathbb{E}\bigg[h^{-1} \sum_{i=1}^{nh} \big(\Delta \tilde{X}_{ik}\big)^{\otimes 2}\bigg] = h^{-1} \int_{I_{k}} \Sigma(t) \inte t = \Sigma_{\avg}^{kh}. 
        $$
        \item Let $t_{i, k} := kh + i/n$. The increments $\Delta \check{X}_{ik}$ form a martingale difference sequence with respect to the canonical filtration $\mathcal{G}^{W}_{t} := \sigma(W(s),\ s \leq t)$, i.e.\
        $$
            \EW{\big\Vert \Delta \check{X}_{ik} \big\Vert} < \infty \text{ and } \EWcond{\Delta \check{X}_{ik}}{\mathcal{G}^{W}_{t_{i-1, k}}} = 0 \text{ almost surely}. 
        $$
        For a deeper discussion of martingale difference sequences, we refer the reader to \cite[Chapter 16]{DavidsonJames2021Slt:}. 
    \end{enumerate}
\end{remark}

\subsection{Rates for the population version of the test statistics}\label{sec:poph0}
As the next step towards deriving the critical value $\kappa_{\alpha}$ in \eqref{eq:Test}, we provide a general bound for the population version of the test statistics $T_{n,h}$ defined in \eqref{eq:Statistic}. 

\indent The behaviour of test statistics $T_{n,h}$ from \eqref{eq:Statistic} exhibits a bias-variance trade-off depending on the block length $h \in (0,1]$. In particular, a large $h$ may artificially inflate the $(r+1)$th largest eigenvalue. Even if $\lambda_{r+1}(\Sigma(t)) = 0$ for all $t \in [0,1]$, the averaged covariance $\Sigma_{\avg}^{kh} = h^{-1}\int_{I_{k}}\Sigma(t) \inte t$, $k \in \ldbrack 0, h^{-1}-1\rdbrack$, may still satisfy $\lambda_{r+1}(\Sigma^{kh}_{\avg}) > 0$. This effect follows directly from the variational characterisation of eigenvalues: averaging positive semi-definite matrices can only increase the rank (see for instance \cite[Proposition 1.3.4]{Tao2012} and \cite[Lemma 2.2]{ReissWinkelmann2023}). Thus, to understand the behaviour of our test statistic under the null hypothesis, it is essential to quantify how large the eigenvalue $\lambda_{r+1}(\Sigma_{\avg}^{kh})$ may become under the constraint $\rank(\Sigma(t)) = r$ for all $t \in [0,1]$. 
\begin{lemma} \label{cor:MCor3.2} 
    For any $\Sigma \in \mathbb{R}^{d \times d}_{\spd}$ with maximal rank $r \in \ldbrack 0, d-1 \rdbrack$ we have 
   \begin{align*}
        \sum_{k=0}^{h^{-1}-1}h\lambda_{r+1}(\Sigma_{\avg}^{kh}) \leq \begin{cases}
             L_{\Sigma}h^{\beta_{\Sigma}} &\text{under $\mathcal{H}_0(r, \beta_b, \beta_{\Sigma}, L_b,  L_{\Sigma})$},\\
             2\underline{\lambda}_r^{-1} L_{\Sigma}^2h^{2\beta_{\Sigma}} &\text{under $\mathcal{H}_0^{\mathrm{gap}}(r, \beta_b, \beta_{\Sigma}, L_b, L_{\Sigma}, \underline{\lambda}_r)$.}
        \end{cases}
    \end{align*}
\end{lemma}
\begin{proof}
    This can be found in \cite[Corollary 3.2]{ReissWinkelmann2023}.
\end{proof}

Combining the eigenvalue control from the preceding lemma with the bounds established in the previous subsection, we obtain an upper bound for the population version of the test statistic.

\begin{theorem}\label{thm:UBad}
    Let $\Sigma$ be a symmetric matrix with maximal rank $r \in \ldbrack0, d-1 \rdbrack$ on $I_k$, $k \in \ldbrack 0, h^{-1}-1\rdbrack$. Grant $1 \geq \beta_b \geq \beta_{\Sigma} > 0$. For the population version of the test statistics $T_{n,h}$ in \eqref{eq:Test} this yields 
    $$\EW{T_{n,h}} \lesssim \begin{cases}
             h^{\beta_{\Sigma}} \big(L_{\Sigma} + L_b^2 n^{-1}h^{\beta_{b}}\big) &\text{under $\mathcal{H}_0 \left(r, \beta_b, \beta_{\Sigma}, L_b, L_{\Sigma}\right)$},\\
             h^{2\beta_{\Sigma}}\big(\underline{\lambda}_r^{-1}L_{\Sigma}^2 + L_b^2 n^{-1}\big) &\text{under $\mathcal{H}_0^{\mathrm{gap}}(r, \beta_b, \beta_{\Sigma}, L_b, L_{\Sigma}, \underline{\lambda}_r)$}.
        \end{cases} $$
\end{theorem}
\begin{remark}
    The mild assumption $\beta_b \geq \beta_{\Sigma}$ ensures that the drift process $b$ is at least as smooth as the deterministic covariance function $\Sigma$. If $\beta_b < \beta_{\Sigma}$, the same proof strategy applies, with the only difference arising in the final step, where the leading rate can no longer be identified. Since the qualitative behaviour remains unchanged, we impose this assumption for simplicity.
\end{remark}
\begin{proof}[Proof of Theorem~\ref{thm:UBad}]
    Combining \eqref{eq:TUB} with \eqref{Eq:EigenvalueProjection} and the results in Propositions~\ref{prop:biash0ad} and \ref{prop:stocherrorh0ad}, it follows that 
    \begin{align*}
        \EW{T_{n,h}} &\leq \sum_{k=0}^{h^{-1}-1} h \big(\big\Vert\EWest{\hat{\Sigma}_{>r}^{kh}} \big\Vert + \EWest{\big\Vert \hat{\Sigma}_{>r}^{kh} - \EWest{\hat{\Sigma}_{>r}^{kh}}\big\Vert}\big) \\
        &\lesssim \bigg(\sum_{k=0}^{h^{-1}-1} h \lambda_{r+1}\big(\Sigma_{\avg}^{kh}\big)\bigg) + L_b^2 n^{-1}h^{2\beta_b}.
    \end{align*}
    An application of Lemma~\ref{cor:MCor3.2} yields 
    \begin{align*}
        \EW{T_{n,h}} 
        &\lesssim \begin{cases}
                 L_{\Sigma}h^{\beta_{\Sigma}} + L_b^2 n^{-1}h^{2\beta_b} &\text{under $\mathcal{H}_0(r, \beta_b, \beta_{\Sigma}, L_b, L_{\Sigma})$},\\
                 \underline{\lambda}_r^{-1}L_{\Sigma}^2h^{2\beta_{\Sigma}} + L_b^2 n^{-1}h^{2\beta_b} &\text{under $\mathcal{H}_0^{\mathrm{gap}}(r, \beta_b,  \beta_{\Sigma}, L_b, L_{\Sigma}, \underline{\lambda}_r)$.}
            \end{cases}
    \intertext{The assumption $\beta_b \geq \beta_{\Sigma}$, i.e.\ the drift process $b$ is smoother than the covariance function $\Sigma$, implies $h^{\beta_b} \leq h^{\beta_{\Sigma}}$, and hence}
         \EW{T_{n,h}} &\lesssim \begin{cases}
                 h^{\beta_{\Sigma}} \big(L_{\Sigma} + L_b^2 n^{-1}h^{\beta_{b}}\big) &\text{under $\mathcal{H}_0(r, \beta_b, \beta_{\Sigma}, L_b, L_{\Sigma})$},\\
                 h^{2\beta_{\Sigma}}\big(\underline{\lambda}_r^{-1}L_{\Sigma}^2 + L_b^2 n^{-1}\big) &\text{under $\mathcal{H}_0^{\mathrm{gap}}(r, \beta_b, \beta_{\Sigma}, L_b, L_{\Sigma}, \underline{\lambda}_r)$},
            \end{cases}
    \end{align*}
    which proves the theorem. 
\end{proof}

\begin{remark}
    In the drift-free setting, the second moment estimator attains the same leading convergence rates (see \cite[Theorem 3.3]{ReissWinkelmann2023}). The presence of a drift affects the asymptotic behaviour only when it is non-constant. In this case, i.e.\ when $L_b > 0$, an additional drift-induced term appears in the rate, which vanishes as $n \rightarrow \infty$ and $h \rightarrow 0$. 
    \indent These results are based on prior knowledge of the regularity of $b$ and $\Sigma$ and the existence of a spectral gap $\underline{\lambda}_r$ given by the null hypothesis.
\end{remark}

The results so far allow to control the level-$\alpha$ test $\varphi_{\alpha}$ under the null hypotheses without and with a spectral gap. We can derive non-asymptotic critical values. 

\begin{theorem}\label{thm:H0adapted}
    Grant $1 \geq \beta_b \geq \beta_{\Sigma} > 0$ and Assumption~\ref{ass:nh}. Fix $\alpha \in (0,1)$ and consider the test $\varphi_{\alpha}$ from \eqref{eq:Test} with critical value $\kappa_\alpha$ defined for suitable constants $c_1,c_2 > 0$ as follows: 
    \begin{enumerate}
        \item Without a minimal spectral gap, take the critical value  
        $$
            \kappa_{\alpha} = h^{\beta_{\Sigma}}\left(c_1 L_{\Sigma} + c_2 L_b^2 n^{-1}h^{\beta_b}\right) \alpha^{-1}.
        $$ 
        Then $\varphi_{\alpha}$ has level $\alpha$ uniformly over $\mathcal{H}_0(r,\beta_b, \beta_{\Sigma}, L_b, L_{\Sigma})$: 
        $$
            \sup_{(\Sigma, b) \in \mathcal{H}_0(r,\beta_b, \beta_{\Sigma}, L_b, L_{\Sigma})} \PPindt{(\Sigma, b)}{\varphi_{\alpha} = 1} \leq \alpha.
        $$ \label{thm:critadi}
        \item Assume the minimal spectral gap $\underline{\lambda}_r>0$. With critical value 
        $$
            \kappa_{\alpha} = h^{2\beta_{\Sigma}}\left(c_1  \underline{\lambda}_r^{-1} L_{\Sigma}^2 + c_2 L_b^2 n^{-1}\right)\alpha^{-1},
        $$ 
        the test $\varphi_{\alpha}$ has level $\alpha$ uniformly over $\mathcal{H}_0^{\mathrm{gap}}(r, \beta_b, \beta_{\Sigma}, L_b, L_{\Sigma}, \underline{\lambda}_r)$: 
        $$
            \sup_{(\Sigma, b) \in \mathcal{H}_0^{\mathrm{gap}}(r,\beta_b, \beta_{\Sigma}, L_b, L_{\Sigma}, \underline{\lambda}_r)} \PPindt{(\Sigma, b)}{\varphi_{\alpha} = 1} \leq \alpha.
        $$ \label{thm:critadii}
    \end{enumerate}
\end{theorem}
\vspace{-0.7cm}
\begin{proof}
    To establish Part~\ref{thm:critadi}, consider $(\Sigma, b) \in \mathcal{H}_0\left(r, \beta_b, \beta_{\Sigma}, L_b, L_{\Sigma}\right)$. Markov's inequality yields
    \begin{align*}
    \PPind{(\Sigma, b)}{\varphi_{\alpha} =1} &= \PPind{(\Sigma, b)}{T_{n,h}>\kappa_{\alpha}}\leq \kappa_{\alpha}^{-1}\EWind{(\Sigma, b)}{T_{n,h}}.
    \intertext{Assuming $1 \geq \beta_b \geq \beta_{\Sigma} > 0$ and applying Theorem~\ref{thm:UBad}, it follows that} 
        \PPind{(\Sigma, b)}{\varphi_{\alpha} =1} &\leq \kappa_{\alpha}^{-1}\big(h^{\beta_{\Sigma}} \big(c L_{\Sigma}  + c^{\prime} L_b^2 n^{-1}h^{\beta_b}\big)\big), 
    \end{align*}
    where we adjusted the constants $c,\ c^{\prime} > 0$ to conclude Part~\ref{thm:critadi}.
    The claim is an immediate consequence of the choice of $\kappa_{\alpha}$ in the theorem.
    
    \indent The spectral gap case~\ref{thm:critadii} follows by the same arguments. 
\end{proof}

\begin{remark}
    The test $\varphi_{\alpha}$ in \eqref{eq:Test} controls the level $\alpha$ in a non-asymptotic sense, but is conservative because the numerical constants are derived from less precise upper bounds. For this reason, we state only the rates determined by the corresponding H\"older-regularity conditions and omit the constants, which can be found in the proofs. Moreover, the critical values grow proportionally to $d-r$ as the dimension increases. In genuinely high-dimensional settings, such scaling may become suboptimal, and additional structural assumptions on $\Sigma$ would be required to avoid ill‑posedness of the testing problem. 
    
    \indent Despite the conservative nature of the constants, the structural composition of the critical value becomes transparent. The terms $L_{\Sigma}h^{\beta_{\Sigma}}$ and $\underline{\lambda}_r^{-1}L_{\Sigma}^2h^{2\beta_{\Sigma}}$ result
    from Lemma~\ref{cor:MCor3.2} depending on the presence or absence of a spectral gap $\underline{\lambda}_r$, respectively. The inclusion of a (non-constant) drift term contributes additional components of order $n^{-1}h^{\beta_b}$ and $n^{-1}$, depending on whether a spectral gap exists. All remaining contributions are, under our assumptions, dominated by these principal terms and are therefore negligible in the overall scaling.
\end{remark}

\subsection{Power of the test} \label{sec:alternative}
In this section, we analyse the power of the test $\varphi_{\alpha}$ and show that, under suitable smoothness assumptions on $b$ and $\Sigma$, it is consistent as $n \to \infty$ for local alternatives whose $(r+1)$th eigenvalues exceed the chosen critical values. In other words, even if the $(r+1)$th eigenvalues of the spot covariance matrix $\Sigma$ tend to zero, the test still achieves asymptotic power one, provided they remain above some critical value. This establishes what is known as a \emph{separation rate} between null hypotheses and alternatives in non-parametric statistics, or equivalently a \emph{detection rate} in signal processing and learning. 

\indent Analogously to \cite{ReissWinkelmann2023}, in our setting it is not possible to define the set of local alternatives by imposing the condition $\int_{0}^{1} \lambda_{r+1}(\Sigma(t)) \inte t \geq v_n$ for some sequence $(v_n)_{n \in \mathbb{N}}$ representing the detection rate. Due to the Riemann-type sums underlying the test statistics $\varphi_\alpha$, such a global $\Lp{1}$-type requirement is too strong, as it is overly sensitive to highly oscillating and spike-like deviations that do not include a stable signal contribution in the discrete approximation underlying the statistic. We therefore follow them and adopt a weaker notion of deviation from zero. Instead of controlling the total integrated deviation globally, we assume that a rectangle of area $v$ can be placed between the graph of $\lambda_{r+1}(\Sigma(t))$ and the $t$-axis. This condition formalises the requirement that the function exhibits a sufficiently persistent deviation from zero on some interval, in the sense that both the magnitude and the length of the deviation jointly exceed a minimal detection level. The required condition can be interpreted as a localised, interval-based analogue of a weak $\Lp{1}$-norm over intervals. 

\indent Building on this notion of a localised deviation from zero, we now formally define the corresponding class of local alternatives in the next definition.

\begin{definition}
    For $r \in \ldbrack 0, d-1\rdbrack$, $\beta_b \in(0,1]$, $L_b \geq 0$, $\hbar \in (0,1)$ and $v,C > 0$ consider the set of alternatives 
    \begin{align*}
        &\mathcal{H}_1(r, \beta_b, L_b, \hbar, v, C) \\
        &:= \bigg\{ (\Sigma, b) \in C \big([0,1]; \mathbb{R}^{d \times d}_{\spd}\big) \times \mathcal{H}^{\beta_b}(L_b) \, \bigg| \, \sup_{I \subseteq [0,1], \, |I| \geq \hbar} |I| \min_{t \in I} \lambda_{r+1}\left(\Sigma(t)\right) \geq v, \tr\bigg(\int_{0}^{1} \Sigma(t) \inte t\bigg) \leq C \bigg\},
    \end{align*} 
    where $|I|$ denotes the length of the interval $I \subseteq [0,1]$. 
\end{definition}  

\begin{remark}
    The alternative $\mathcal{H}_1$ neither involves a $\beta_{\Sigma}$-smoothness constraint nor a spectral gap condition. 
\end{remark}

In the following, we make use of the decomposition $\Delta X_{ik} = \Delta \tilde{X}_{ik} + B_{ik}$, where $\Delta \tilde{X}_{ik}$ and $B_{ik}$ are defined as in Notation~\ref{defi:form}, for $i \in \ldbrack 1, nh \rdbrack$ and $k \in \ldbrack 0, h^{-1}-1 \rdbrack$. For notational convenience, we write $\mathbb{E}_{(\Sigma, b)}$ and $\mathbb{P}_{(\Sigma, b)}$ for expectation and probability under the alternative indexed by $(\Sigma, b)$. Furthermore, to allow for local alternatives, we introduce an $n$-dependence in both the level-$\alpha$ test and the block length by considering $\varphi_{\alpha,n}$ and $h_n$, respectively.

First, we establish that the test is consistent under the stated assumptions. 

\indent For the following theorem and corollary, we denote by $C_{d,r+1} > 1$ the constant from Corollary S.9 in \cite{ReissWinkelmann2023}, which depends on $d$ and $r+1$. 

\begin{theorem}\label{thm:h1}
    Consider for sample size $n$ the asymptotic test $\varphi_{\alpha,n}$ from Equation \eqref{eq:Test} with block sizes $h_n$ and critical values 
    $$\kappa_{\alpha,n} = c_{n}\big(n^{-1}h_n^{-2} + L_b n^{-1/2} h_n^{\beta_b - 1/2}\big)$$ 
    for some sequence $(c_{n})_{n \in \mathbb{N}}$ such that $c_n \rightarrow \infty$ for $n \rightarrow \infty$. Assume $h_n \rightarrow 0$ as $n \rightarrow \infty$ and that the number of observations per block satisfies $nh_n \geq 2(r+1)C_{d,r+1}$ with a fixed constant $C_{d,r+1}>1$ only depending on $d$ and $r+1$. Then, $\varphi_{\alpha,n}$ is asymptotically consistent over the local alternatives $\mathcal{H}_1(r, \beta_b, L_b, \hbar_n, v_n, C)$, provided $\hbar_n / h_n \rightarrow \infty$, and the rate $v_n$ is larger than a constant multiple of $\kappa_{\alpha,n}$: 
    $$
        \exists c^{\prime}>0: \lim_{n \rightarrow \infty} \inf \left\{\PPind{(\Sigma, b)}{\varphi_{\alpha,n} = 1} = 1 \left| (\Sigma,b) \in \mathcal{H}_1(r, \beta_b, L_b, \hbar_n, c^{\prime} \kappa_{\alpha,n}, C)\right.\right\}.
    $$
\end{theorem}
\begin{proof}
    The proof is based on the following three-step strategy. First, Bonferroni's inequality is applied to derive a lower bound for the probability of interest. Second, convergence results for the resulting probabilities are established. Finally, the resulting convergence results are combined to conclude the proof of the theorem. The corresponding arguments are formalised in auxiliary lemmas stated and proved in Appendix~\ref{a:alternative}. In the proof below, we therefore only verify that the assumptions of these lemmas are satisfied in the present setting and subsequently combine the resulting upper bounds.
    
    Throughout the proof, we denote $\mathcal{H}_1(r, \beta_b, L_b, \hbar_n, c^{\prime} \kappa_{\alpha,n}, C)$ briefly by $\mathcal{H}_1$. Let $\check{\kappa}_{\alpha,n}$ and $\tilde{\kappa}_{\alpha,n}(\Sigma)$ be as introduced in Lemmas~\ref{lem:alternativeMR} and \ref{lem:alternative2ndterm} in the Appendix~\ref{a:alternative}. Since $(\Sigma,b) \in \mathcal{H}_1$, the quantity $\tilde{\kappa}_{\alpha,n}(\Sigma)$ introduced in Lemma~\ref{lem:alternative2ndterm} can be uniformly bounded over $(\Sigma, b) \in \mathcal{H}_1$. In particular, its dependence on $\Sigma$ is asymptotically negligible, so that $\tilde{\kappa}_{\alpha,n}(\Sigma) = \tilde{c} \kappa_{\alpha,n}$ for a sufficiently large constant $\tilde{c}>0$. Note that Lemma~\ref{lem:alternative2ndterm} applies, and by monotonicity of the power in the critical value, the power can only decrease when larger critical values are used. Moreover, Lemma~\ref{lem:alternativeMR} holds for some $\check{\kappa}_{\alpha,n}$, which allows us to fix $\check{\kappa}_{\alpha,n} = \check{c}\kappa_{\alpha,n}$ for some constants $\check{c} > 0$.
    Combining these with Lemma~\ref{lem:decomph1}, which provides the required decomposition via an application of Bonferroni's inequality, we obtain  
    \begin{align*}
        &\liminf_{n \rightarrow \infty} \ \inf_{(\Sigma,b) \in \mathcal{H}_1} \PPind{(\Sigma,b)}{\varphi_{\alpha,n} = 1} \\
        &\qquad \geq \liminf_{n \rightarrow \infty} \ \inf_{(\Sigma,b) \in \mathcal{H}_1} \bigg[ \mathbb{P}_{(\Sigma,b)}\bigg(\sum_{k=0}^{h_n^{-1}-1} \lambda_{r+1}\bigg(\sum_{i=1}^{nh_n} \big(\Delta \tilde{X}_{ik}\big)^{\otimes 2}\bigg) \geq c \kappa_{\alpha,n}\bigg)\\
        &\phantom{\geq} \qquad \qquad \qquad \qquad \qquad \qquad \qquad \qquad \qquad \qquad \qquad \qquad - \mathbb{P}_{(\Sigma,b)}\bigg(\sum_{k=0}^{h_n^{-1} - 1} h_n \big\Vert \mathcal{R}^{kh_n} \big\Vert \geq \check{c}\kappa_{\alpha,n}\bigg) \bigg] \\
        &\qquad\geq \liminf_{n \rightarrow \infty} \bigg[ \inf_{(\Sigma,b) \in \mathcal{H}_1} \mathbb{P}_{(\Sigma,b)}\bigg(\sum_{k=0}^{h_n^{-1}-1} \lambda_{r+1}\bigg(\sum_{i=1}^{nh_n} \big(\Delta \tilde{X}_{ik}\big)^{\otimes 2}\bigg) \geq c \kappa_{\alpha,n}\bigg) \\
        &\phantom{\geq}\qquad\qquad \qquad \qquad \qquad \qquad \qquad\qquad\qquad\qquad-\sup_{(\Sigma,b) \in \mathcal{H}_1}\mathbb{P}_{(\Sigma,b)}\bigg(\sum_{k=0}^{h_n^{-1} - 1} h_n \big\Vert \mathcal{R}^{kh_n} \big\Vert \geq \check{c}\kappa_{\alpha,n}\bigg) \bigg], 
    \end{align*}
    where $c := \tilde{c} + \check{c} > 0$. We deduce from Lemmas~\ref{lem:alternativeMR} and \ref{lem:alternative2ndterm}
    \begin{align*}
        \lim_{n \rightarrow \infty} \ \inf_{(\Sigma, b) \in \mathcal{H}_1} \PPind{(\Sigma, b)}{\varphi_{\alpha,n} = 1} = 1,
    \end{align*}
    which proves the theorem. 
\end{proof}

\begin{cor}\label{cor:h1}
    Consider the test $\varphi_{\alpha, n}$ defined in Equation \eqref{eq:Test} for sample size $n$, with block sizes $h_n$ and critical value $\kappa_{\alpha,n}$ from Theorem~\ref{thm:H0adapted}. Assume that $1 \geq \beta_{b} \geq \beta_{\Sigma} > 0$, $h_n \rightarrow 0$ as $n \rightarrow \infty$ and that the number of increments per block satisfies $n h_n \geq 2(r+1)C_{d,r+1}$. 
    \begin{enumerate}
        \item Without assuming a minimal spectral gap, if $L_b > 0$, choose 
        $$
            h_n \sim \left(L_{\Sigma}n \right)^{-1/(\beta_{\Sigma} + 2)} \vee \big(L_b^{-1} L_{\Sigma} n^{1/2}\big)^{-1/(\beta_{\Sigma}- \beta_b + 1/2)}.
        $$
        If $L_b = 0$, choose $h_n \sim \left(L_{\Sigma}n \right)^{-1/(\beta_{\Sigma} + 2)}$.
        \item Assume the minimal spectral gap $\underline{\lambda}_{r,n} > 0$, which may vary arbitrarily with $n$. If $L_b > 0$, choose 
        $$
            h_n \sim \left(\underline{\lambda}_{r,n}^{-1} L_{\Sigma}^2 n\right)^{-1/(2 \beta_{\Sigma} + 2)} \vee \big(\underline{\lambda}_{r,n}^{-1} L_{b}^{-1} L_{\Sigma}^2  n^{1/2}\big)^{-1/(2\beta_{\Sigma} - \beta_{b} + 1/2)}. 
        $$
        If $L_b = 0$, choose $h_n \sim \left(\underline{\lambda}_{r,n}^{-1} L_{\Sigma}^2 n\right)^{-1/(2 \beta_{\Sigma} + 2)}$. 
    \end{enumerate}
    Then, $\varphi_{\alpha,n}$ is asymptotically consistent over the local alternatives $\mathcal{H}_1(r, \beta_b, L_b, \hbar_n, v_n, C)$ provided $\hbar_n / h_n \rightarrow \infty$ and the rate $v_n$ is larger than a constant multiple of $\kappa_{\alpha,n}$ given in Theorem~\ref{thm:h1}: 
    $$
        \exists c^{\prime}>0: \lim_{n \rightarrow \infty} \inf\left\{\PPind{(\Sigma, b)}{\varphi_{\alpha,n} = 1} = 1 \left| (\Sigma,b) \in \mathcal{H}_1(r, \beta_b, L_b, \hbar_n, c^{\prime} \kappa_{\alpha,n},C)\right.\right\}.
    $$
\end{cor}
\begin{proof}
    In the proof we denote $\mathcal{H}_1(r, \beta_b, L_b, \hbar_n, c^{\prime} \kappa_{\alpha,n}, C)$ briefly by $\mathcal{H}_1$. Under the assumption that $(\Sigma, b) \in \mathcal{H}_1$, the critical value $\kappa_{\alpha,n}$ from Theorem~\ref{thm:h1} is of the same order as the critical values obtained in Theorem~\ref{thm:H0adapted}. Therefore, by Lemmas~\ref{lem:decomph1}--\ref{lem:alternative2ndterm} we have
    $$
        \lim_{n \rightarrow \infty} \ \inf_{(\Sigma,b) \in \mathcal{H}_1} \PPind{(\Sigma,b)}{\varphi_{\alpha,n} = 1} = 1,
    $$
    which completes the proof.
\end{proof}
\newpage
\begin{remark} 
    Assume that $1 \geq \beta_{b} \geq \beta_{\Sigma} > 0$ and $L_{b} \geq 0,\ L_{\Sigma} > 0$. 
    \begin{enumerate}
    \item Without a minimal spectral gap and assume $L_b > 0$, the detection rate is given by 
    $$
        v_n = c \left(\left(L_{\Sigma}^{2} n^{-\beta_{\Sigma}}\right)^{1/(\beta_{\Sigma} + 2)} \wedge \big(L_b^{\beta_{\Sigma}} L_{\Sigma}^{-(\beta_{b} - 1/2)} n^{-\beta_{\Sigma}/2}\big)^{1/(\beta_{\Sigma} - \beta_{b} + 1/2)}\right)
    $$
    for a sufficiently large constant $c > 0$, and with $L_b > 0$. In the special case $L_b = 0$, the block length $h_n$ from Corollary~\ref{cor:h1} satisfies $h_n \sim (L_{\Sigma} n)^{-1/(\beta_{\Sigma} + 2)}$, which implies $v_n = c (L_{\Sigma}^2 n^{-\beta_{\Sigma}})^{1/(\beta_{\Sigma} + 2)}$, for a (possibly different) sufficiently large constant $c > 0$. 
    \item Suppose that the minimal spectral gap $\underline{\lambda}_{r,n} > 0$ exists, which may vary arbitrarily with $n$, and assume further that $L_b > 0$. Then, the detection rate $v_n$ is given by 
    $$
         v_n = c \left(\big(\underline{\lambda}_{r,n}^{-1} L_{\Sigma}^{2} n^{-\beta_{\Sigma}}\big)^{1/( \beta_{\Sigma} + 1)} \wedge \big(\big(\underline{\lambda}_{r,n}^{-1} L_{\Sigma}^{2}\big)^{-(\beta_{b} - 1/2)}  \left(L_{b}^{-2} n\right)^{-\beta_{\Sigma}}\big)^{1 /(2 \beta_{\Sigma} - \beta_{b} + 1/2)}\right),
    $$
    where $c > 0$ denotes a sufficiently large constant. In the case $L_b = 0$, the block length $h_n$ from Corollary~\ref{cor:h1} satisfies $h_n \sim (\underline{\lambda}_{r,n}^{-1} L_{\Sigma} n)^{-1/(2 \beta_{\Sigma} + 2)}$, which yields $v_n = c (\underline{\lambda}_{r,n}^{-1}L_{\Sigma}^2n^{-\beta_{\Sigma}})^{1/(\beta_{\Sigma} + 1)}$ for a (possibly different) sufficiently large constant $c>0$. 
    \end{enumerate}
\end{remark}   

\begin{example}
    To compare the detection rates obtained in Corollary~\ref{cor:h1} with those in \cite[Corollary 4.3]{reiss2026rank}, we restrict attention to the case of a constant drift term, i.e.\ $L_b = 0$. In line with Example 4.4 therein, we assume $\beta_{\Sigma} = 0.5$ and $L_{\Sigma} \sim 1$. Under these assumptions, we choose the block length $h \sim n^{-2/5} \vee (\underline{\lambda}_{r}^{-1} n)^{-1/3}$, which leads to the detection rate $v_n \sim n^{-1/5} \wedge (\underline{\lambda}_{r}^{-2/3}n^{-1/3})$. These rates coincide with those reported in their example. For further interpretation, we refer to the discussion therein. 
\end{example}

\section{Numerical simulation}\label{sec_nu}
In this section, we extend Example 3.10 in \cite{ReissWinkelmann2023} to analyse the impact of the drift term and the magnitude of the second eigenvalue of the spot covariance matrix on the power behaviour of rank tests. Our focus lies on comparing the test based on the re-centred covariance estimator with the test based on the second moment estimator. The simulations of the underlying stochastic differential equation are based on the Euler--Maruyama scheme with time stepping $\Delta t = (50n)^{-1}$. 

\begin{example} \label{ex:3.10MLLinDrift}
    We consider a two-dimensional It\^o semi-martingale  with a linear drift. Define 
    $$b(t) := \begin{pmatrix}
        -20 t \\ -100t
    \end{pmatrix}, \ v_1 (t) := \begin{pmatrix}\underline{\lambda}_1^{1/2}\\ (h^{2 \beta_{\Sigma}}/\underline{\lambda}_1)^{1/2}\sin(2 \pi t /h)\end{pmatrix} \text{, } v_2(t) := \begin{pmatrix} (h^{2 \beta_{\Sigma}}/\underline{\lambda}_1)^{1/2} \sin(2 \pi t/h)\\ - \underline{\lambda}_1^{1/2}\end{pmatrix}.$$ 
    We set $\Sigma(t) := v_1(t)v_1(t)^\top + \gamma v_2(t)v_2(t)^\top$, where $\gamma \in [0,1]$. Since $v_1(t)$ and $v_2(t)$ are orthogonal, they form an eigenbasis of $\Sigma(t)$ with eigenvalues $\lambda_1(\Sigma(t)) = \underline{\lambda}_1 + (h^{2 \beta_{\Sigma}}/\underline{\lambda}_1) \sin^2(2 \pi t/h)$ and $\lambda_2(\Sigma(t)) = \gamma \lambda_1(\Sigma(t))$. 
    
    \indent For $\gamma \in (0,1]$, we generate spot covariance matrices under the alternative $\mathcal{H}_1$ : $r=2$ corresponding to signal strengths $h^{-1} \int_{0}^{h} \lambda_2(\Sigma(t)) \inte t = \gamma(\underline{\lambda}_1 + h^{2 \beta_{\Sigma}} / (2 \underline{\lambda}_1))$. 
\end{example} 

\indent For $\beta_{\Sigma} = 0.5$ Figure~\ref{fig:LinDriftEx320_beta05} displays contour plots of the power for the test based on the second moment estimator $\hat{\Sigma}^{kh}_{2\mathrm{nd}}$, defined in \eqref{eq:secmomest}, i.e. 
\begin{align}\label{eq:testsec}
    \varphi_{\alpha}^{\textrm{2nd}} := \indikator\bigg(\sum_{k=0}^{h^{-1}-1} h \lambda_{r+1}\big(\hat{\Sigma}^{kh}_{\textrm{2nd}}\big) > \kappa_{\alpha}^{\textrm{2nd}} \bigg),
\end{align}
where $\kappa_{\alpha}^{\textrm{2nd}}$ denotes the critical value corresponding to $\alpha > 0$, and for the test based on the re-centred covariance estimator $\hat{\Sigma}^{kh}$, defined in \eqref{eq:Test}. 

For each pair $(\underline{\lambda}_1, \text{signal strength})$, we compute the empirical power, that is, the probability that the level-$\alpha$ test rejects the null hypothesis at significance level $\alpha = 0.1$. The power is reported as a function of the spectral gap $\underline{\lambda}_1$ (horizontal axis) and the signal strength, i.e.\ the time-averaged second eigenvalue, (vertical axis). The simulations are based on $n = 1{,}000$ observations and $1{,}000$ Monte Carlo replications with $h=0.02$, corresponding to a local sample size of $nh = 20$. 

\indent The critical values for both estimators are obtained from $1{,}000$ Monte Carlo replications for each value of $\underline{\lambda}_1$. As shown in Table~\ref{tab:crit}, the critical values $\kappa_{0.1}$ are uniformly smaller than $\kappa_{0.1}^{\mathrm{2nd}}$. The covariance-based test \eqref{eq:Test} displays a faster increase in power, achieving power close to one already for smaller signal strength, while the second moment-based test \eqref{eq:testsec} requires larger signals. Overall, the covariance-based test \eqref{eq:Test} dominates the second moment-based test \eqref{eq:testsec} over the entire parameter range considered with regards to power. In both cases, the power is increasing in the signal strength as well as in the spectral gap.  

\begin{figure}[t]
    \centering
    \includegraphics[width=\linewidth]{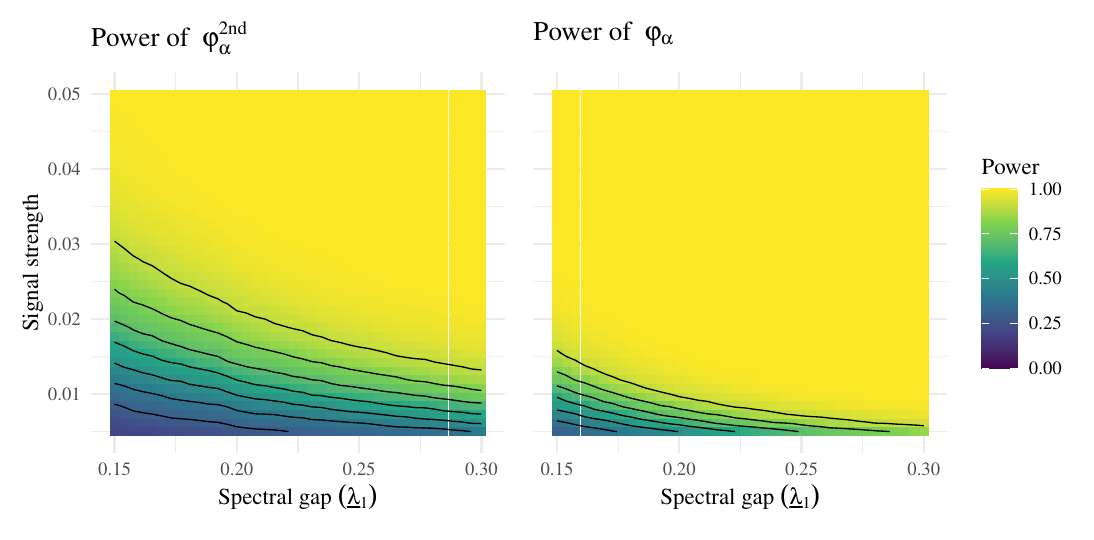}
    \caption{Power of the two rank tests as a function of signal strength and spectral gap $\underline{\lambda}_1$ in Example~\ref{ex:3.10MLLinDrift} with $\beta_{\Sigma} = 0.5$. The black curves represent contour lines of the power function corresponding to the levels $0.1, 0.2, \ldots, 0.9$. \emph{Left:} Rank test based on the second moment estimator given by \eqref{eq:testsec}. \emph{Right:} Rank test based on the covariance estimator given by \eqref{eq:Test}.}
    \label{fig:LinDriftEx320_beta05}
\end{figure}

\begin{table}[h]
    \centering
    \caption{Representative subset of critical values for $\kappa_{0.1}^{\textrm{2nd}}$ and $\kappa_{0.1}$, computed for Example~\ref{ex:3.10MLLinDrift} with $\beta_{\Sigma} = 0.5$. The reported values are rounded to four decimal places and multiplied by $1{,}000$.}
    \label{tab:crit}
    \begin{tabular}{ccccc}
    \toprule
    Spectral gap $(\underline{\lambda}_1)$ & \hspace{1.5cm} & $1{,}000 \cdot \kappa_{0.1}^{\textrm{2nd}}$ & \hspace{1.5cm} & $1{,}000 \cdot \kappa_{0.1}$ \\
    \midrule
    0.15 & & 1.9916 & & 1.1798 \\
    0.20 & & 1.9772 & & 1.9091 \\
    0.25 & & 1.9161 & & 0.7342 \\
    0.30 & & 1.8465 & & 0.6152 \\
    \bottomrule
    \end{tabular}
\end{table}

\begin{figure}[t] \label{fig:varyingN}
    \centering
    \includegraphics[width=\linewidth]{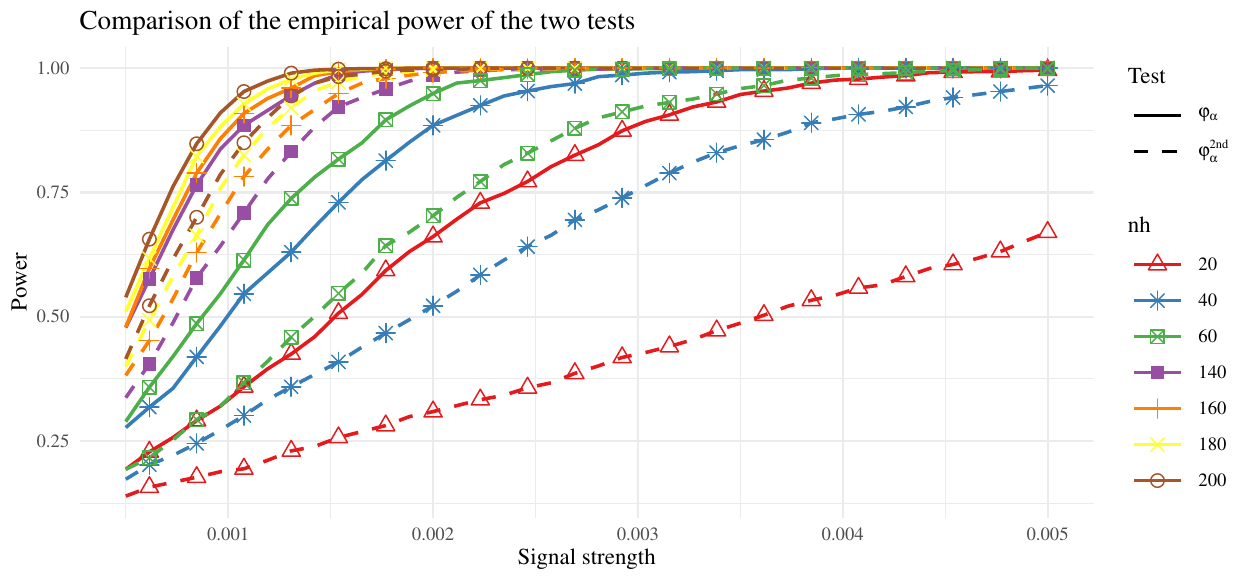}
    \caption{Rejection frequencies of the second moment-based test \eqref{eq:testsec} and covariance-based test \eqref{eq:Test} under $\mathcal{H}_1$ at significance level $\alpha = 0.1$, for block length $h = 0.02$ and varying sample sizes $n$, based on Example~\ref{ex:3.10MLLinDrift} with $1{,}000$ Monte Carlo repetitions.}
\end{figure}

\indent To justify the choice of $1{,}000$ observations in Figure~\ref{fig:LinDriftEx320_beta05}, we conduct an additional power analysis comparing the second moment-based test (\ref{eq:testsec}, dashed lines) and covariance-based test (\ref{eq:Test}, solid lines) at significance level $\alpha = 0.1$. Figure~\ref{fig:varyingN} reports the rejection frequencies under a rank-two alternative for seven different local sample sizes ranging from $20$ to $200$ observations. The number of blocks over the entire interval is fixed at $50$, corresponding to $h=0.02$, and hence are the varying sample sizes given by $n \in\{1{,}000; 2{,}000; 3{,}000; 7{,}000; 8{,}000; 9{,}000; 10{,}000 \}$. The spectral gap is set to $\underline{\lambda}_1 = 0.5$, and we assume $\beta_{\Sigma} = 0.5$, in line with Figure~\ref{fig:LinDriftEx320_beta05}. All results are based on $1{,}000$ Monte Carlo replications for the simulated Example~\ref{ex:3.10MLLinDrift}. The covariance-based test \eqref{eq:Test} exhibits a faster increase in power compared to the second moment-based test \eqref{eq:testsec}. Even as the local sample size increases, the second moment-based test \eqref{eq:testsec} attains comparable but consistently lower power, highlighting the uniformly higher power of the covariance-based test \eqref{eq:Test}.

\indent To quantify the magnitude of a deterministic drift relative to the diffusion component, we consider the following example.

\begin{example}\label{ex:3.10MLVaryingLinDrift}
    We reconsider the extended two-dimensional It\^o semi-martingale introduced in Example~\ref{ex:3.10MLLinDrift} by choosing a varying linear drift component 
    $$
        \tilde{b}(t) := a \begin{pmatrix}
            -2 t \\
        -10 t
        \end{pmatrix}, \, a \in [0,10].
    $$
\end{example}

\begin{figure}[t]
    \centering
    \includegraphics[width=\linewidth]{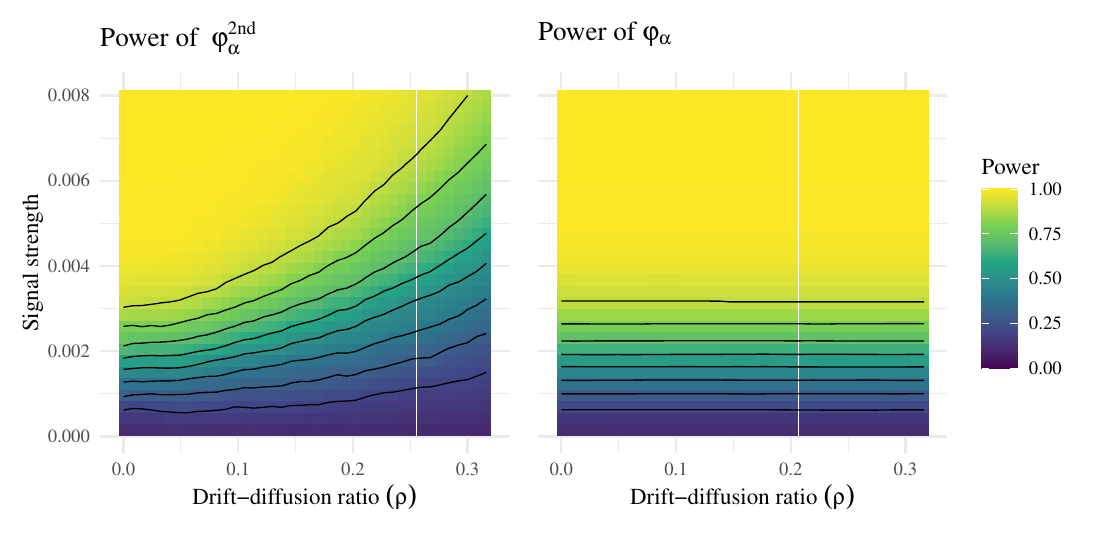}
    \caption{Power of the two rank tests as a function of signal strength and drift-diffusion ratio $\rho$ in Example~\ref{ex:3.10MLVaryingLinDrift} with $\underline{\lambda}_1 = 0.5$ and $\beta_{\Sigma} = 0.5$. \emph{Left:} Rank test at level $\alpha = 0.1$ based on the second moment estimator defined in \eqref{eq:testsec}. \emph{Right:} Rank test at level $\alpha = 0.1$ based on the covariance estimator defined in \eqref{eq:Test}.}
    \label{fig:DifferentSizeLinDrift}
\end{figure}

\begin{table}[h!] \label{tab:critvardrift}
    \centering
    \caption{Representative subset of critical values for $\kappa_{0.1}^{\textrm{2nd}}$ and $\kappa_{0.1}$, computed for Example~\ref{ex:3.10MLVaryingLinDrift} with $\underline{\lambda}_1 = 0.5$ and $\beta_{\Sigma} = 0.5$. The reported values are rounded to five decimal places and multiplied by $1{,}000$.}
    \begin{tabular}{ccccc}
    \toprule
    Drift-diffusion ratio ($\rho$) & \hspace{1.5cm} & $1{,}000 \cdot \kappa_{0.1}^{\textrm{2nd}}$ & \hspace{1.5cm} &$1{,}000 \cdot \kappa_{0.1}$ \\
    \midrule
    0.00 & & 0.39092 & & 0.37178 \\
    0.05 & & 0.42186 & & 0.37178 \\
    0.10 & & 0.51556 & & 0.37180 \\
    0.15 & & 0.70025 & & 0.37181 \\
    0.20 & & 0.92107 & & 0.37182 \\
    0.25 & & 1.20078 & & 0.37184\\
    0.30 & & 1.53330 & & 0.37186 \\
    \bottomrule
    \end{tabular}
\end{table}

We set $\underline{\lambda}_{1} = 0.5$ and choose $\beta_{\Sigma} = 0.5$. The sample size is $n = 1{,}000$ and the block-length equals $h = 0.02$, yielding $nh = 20$ observations per block. All results are based on $1{,}000$ Monte Carlo replications. 

\indent To assess the relative importance of the deterministic drift compared to the diffusion, we introduce the drift-diffusion ratio. Specifically, we compute the average Euclidean norm of the drift and scale it by the instantaneous volatility level, measured by the square root of the trace of the spot covariance matrix. This normalization ensures that the ratio reflects only the relative strength of the drift and diffusion components and remains comparable across time. Concretely, we compute 
\begin{align}\label{eq:rho}
    \rho = \bigg(\sum_{i = 1}^{50n} \big\Vert \tilde{b}(t_i) \big\Vert \bigg)\bigg (\sum_{i=1}^{50n} (\tr(\Sigma(t_i)))^{1/2}\bigg)^{-1} 
\end{align}
with $t_i = i/(50n)$ for all $i \in \ldbrack 1, 50n \rdbrack$.  

\indent Figure~\ref{fig:DifferentSizeLinDrift} depicts the power as a function of the drift-diffusion ratio $\rho$ (horizontal axis) and the signal strength (vertical axis). 

\indent For small values of $\rho$, i.e.\ $\rho < 0.05$, both procedures exhibit comparable detection boundaries. The second moment-based test \eqref{eq:testsec} requires uniformly larger critical values, whereas the critical values of the covariance-based test \eqref{eq:Test} remain approximately constant (see Table~\ref{tab:critvardrift}). As the drift-diffusion ratio increases beyond $\rho = 0.05$, the second moment-based test \eqref{eq:testsec} requires substantially stronger signal strengths to achieve the same power level. In contrast, the covariance-based procedure remains stable across the entire range of $\rho$, which is consistent with its approximately constant critical values.

\subsection*{Acknowledgements}
I want to thank Markus Rei\ss\ for his helpful advice and his support of my work.

\subsection*{Funding}
    This research has been partially funded by Deutsche Forschungsgemeinschaft (DFG) - Project-ID 410208580 - IRTG2544 (``Stochastic Analysis in Interaction''). 

\begin{appendices}\label{app:proofs}

\section{Proofs of Subsection~\ref{sec:mainh0}} \label{a:adapted}
To prove the statements under the null hypothesis, we assume that $b$ is an adapted drift process satisfying $\EWt{\int_{0}^1 \Vert b(t)\Vert \inte t} < \infty$ and Assumption~\ref{ass:hölderad}. Therefore, we state some useful Lemmas. 

\begin{lemma}\label{lem:SumBiBiCond}
    Let $k \in \ldbrack 0, h^{-1}-1 \rdbrack$. Under Assumption~\ref{ass:hölderad}, we have 
    $$\mathbb{E}\bigg[\bigg\Vert\sum_{i=1}^{nh} \Big(B_{ik} - B_{ik}^{\mathrm{cond}}\Big)\bigg\Vert^{2}\bigg] \leq L_b^2 n^{-1} h^{2\beta_b+1}.$$
\end{lemma}
\begin{proof}
    Throughout the proof, for $X=(X_1,\ldots,X_n)^\top \in \Lp{2}(\Omega; \mathbb{R}^{n})$ defined on the probability space $(\Omega, \mathcal{F}, \mathbb{P})$, we use the notation 
    $$
        \Vari{X} := \EWt{\Vert X - \EWt{X} \Vert^2} = \sum_{i=1}^{n} \EWt{(X_i - \EW{X_i})^2} = \sum_{i=1}^{n} \Vari{X_i}
    $$
    (see for instance \cite[Remark A.28]{trabs2021statistik}). Furthermore, for a sub-$\sigma$-algebra $\mathcal G \subseteq \mathcal F$, we write
    $$
        \VariCond{X}{\mathcal{G}} := \EWcond{\Vert X - \EWcond{X}{\mathcal{G}}\Vert^2} {\mathcal{G}} = \sum_{i=1}^n \EWcond{(X_i - \EWcond{X_i}{\mathcal{G}})^2}{\mathcal{G}} = \sum_{i=1}^{n}\VariCond{X_i}{\mathcal{G}}.
    $$
    From now on, let $k \in \ldbrack 0, h^{-1}-1\rdbrack$. By the tower property of conditional expectation, the inner products corresponding to different indices are zero, so that only the squared terms remain. Therefore, we have, by applying Fubini's theorem and the tower property,  
    \begin{align*}
        \mathbb{E}\bigg[\bigg\Vert\sum_{i=1}^{nh} \Big(B_{ik} - B_{ik}^{\mathrm{cond}}\Big)\bigg\Vert^{2}\bigg] 
        &= \sum_{i=1}^{nh} \mathbb{E}\bigg[\left\Vert \int_{I_{ik}} \left(b(s) - \EWcond{b(s)}{\mathcal{F}_{i-1,k}}\right) \inte s \right\Vert^2\bigg] \\
        &= \sum_{i=1}^{nh} \EW{\VariCond{\int_{I_{ik}} b(s) \inte s}{\mathcal{F}_{i-1,k}}} \\
        &\leq \sum_{i=1}^{nh} \Vari{\int_{I_{ik}} (b(s) - b(kh)) \inte s} \\
        &\leq \sum_{i=1}^{nh} \mathbb{E}\bigg[\bigg\Vert \int_{I_{ik}} (b(s) - b(kh)) \inte s \bigg\Vert^2\bigg].
     \end{align*}
    Jensen's inequality and Fubini's theorem yield
    \begin{align*}
        \mathbb{E}\bigg[\bigg\Vert\sum_{i=1}^{nh} \Big(B_{ik} - B_{ik}^{\mathrm{cond}}\Big)\bigg\Vert^{2}\bigg] 
        \leq n^{-1} \sum_{i=1}^{nh} \int_{I_{ik}}  \EW{\left\Vert b(s) - b(kh) \right\Vert^2} \inte s \leq L_b^{2} n^{-1} h^{2 \beta_b + 1},
    \end{align*}
    where Assumption~\ref{ass:hölderad} has been used in the last step. 
\end{proof}

\begin{lemma} \label{lem:BcondAvBcondsquared}
    Let $i \in \ldbrack 1, nh \rdbrack$, $k \in \ldbrack 0, h^{-1}-1 \rdbrack$. Then we have 
    $$
        \mathbb{E}\bigg[\Big\Vert B_{ik}^{\mathrm{cond}} - \overline{ B_{k}^{\mathrm{cond}}}\Big\Vert^2\bigg] \leq 4 L_b^2 n^{-2} h^{2\beta_b}. 
    $$
\end{lemma}
\begin{proof}
    From Lemma~\ref{lem:NormDiffXiMean} with $X_i = \EWcondt{\int_{I_{ik}} (b(s) - b(kh)) \inte s }{\mathcal{F}_{i-1,k}}$, we get 
    \begin{align*}
        \mathbb{E}\bigg[\Big\Vert B_{ik}^{\mathrm{cond}} - \overline{ B_{k}^{\mathrm{cond}}}\Big\Vert^2\bigg]
        &\leq 2 \mathbb{E}\bigg[\left\Vert \mathbb{E}\bigg[\int_{I_{ik}} \left(b(s) - b(kh)\right)  \inte s \bigg| \mathcal{F}_{i-1,k}\bigg]\right\Vert^2\bigg] \\
        &\phantom{\leq}+ 2(nh)^{-1} \sum_{j=1}^{nh}\mathbb{E}\bigg[\bigg\Vert \mathbb{E}\bigg[\int_{I_{jk}} \left(b(s)-b(kh)\right) \inte s \bigg| \mathcal{F}_{j-1,k}\bigg]\bigg\Vert^2\bigg].
    \intertext{Applying $\Vert \EWcondt{X}{\mathcal{H}} \Vert^2 \leq \EWcondt{\Vert X \Vert^2}{\mathcal{H}}$ (cf.\ {\cite[Corollary B.5]{OksendalSDE}}) and Jensen's inequality yields}
        \EW{\Big\Vert B_{ik}^{\mathrm{cond}} - \overline{ B_{k}^{\mathrm{cond}}}\Big\Vert^2}
        &\leq 2(n^2h)^{-1} \sum_{j=1}^{n h} \bigg(\EW{ \mathbb{E}\bigg[\int_{I_{ik}} \left\Vert b(s) - b(kh)\right\Vert^2\inte s \bigg| \mathcal{F}_{i-1,k}}\bigg]  \\
        &\qquad\qquad\qquad\qquad\qquad\phantom{\leq}+ \mathbb{E}\bigg[ \mathbb{E}\bigg[\int_{I_{jk}} \left\Vert b(s) - b(kh)\right\Vert^2\inte s \bigg| \mathcal{F}_{j-1,k}\bigg]\bigg]\bigg) \\
        &\leq 4 L_b^2 n^{-2} h^{2\beta_b},
    \end{align*}
    where Fubini's theorem, the tower property and the H\"older-regularity condition (compare Assumption~\ref{ass:hölderad}) have been used in the last step. 
\end{proof}

\subsection{Proof of Proposition~\ref{prop:biash0ad}}\label{app:biash0ad}
The proof proceeds by first decomposing $h \Vert \EWt{\hat{\Sigma}_{>r}^{kh}} \Vert$ into three terms, which are analysed separately. One of these terms is further decomposed into three subterms to facilitate the derivation of an appropriate upper bound. To this end, we exploit the H\"older-regularity of the drift term $b$, Jensen’s inequality, and Fubini’s theorem, as well as standard properties of centred multivariate Gaussian increments and the trace. Finally, the upper bounds obtained for the individual terms are combined to derive the desired estimate. \\ 
\indent Throughout the proof, let $i \in \ldbrack 1, nh \rdbrack$ and $k \in \ldbrack 0, h^{-1}-1 \rdbrack$. We start by decomposing $\hat{\Sigma}^{kh}$ as 
    \begin{align*}
        \hat{\Sigma}^{kh}
        &= h^{-1} \bigg[\sum_{i=1}^{nh}\left(\Delta \check{X}_{ik} - \overline{\Delta \check{X}_{k}}\right)^{\otimes 2} + \sum_{i=1}^{nh}\left(B_{ik}^{\mathrm{cond}} - \overline{B_{k}^{\mathrm{cond}}}\right)^{\otimes 2} \\
        &\phantom{=}\qquad\qquad\qquad+ \sum_{i=1}^{nh}\left(\Delta \check{X}_{ik} -  \overline{\Delta\check{X}_{k}}\right)\odot\left(B_{ik}^{\mathrm{cond}} - \overline{B_{k}^{\mathrm{cond}}}\right)\bigg],
    \end{align*}
    and hence, by applying the orthogonal projection $\Proj{P}_{>r}$ onto the space $V_{>r}$ generated by the $d-r$ smallest eigenvalues of $\Sigma_{\avg}^{kh}$ to this decomposition, we obtain 
    \begin{align}
         h \big\Vert\EWest{\hat{\Sigma}_{>r}^{kh}}\big\Vert 
         &\leq \left\Vert \EW{\mathcal{T}_{1, >r}} \right\Vert + \left\Vert \EW{\mathcal{T}_{2,>r}} \right\Vert + 2\left\Vert \EW{\mathcal{T}_{3,>r}}\right\Vert, \label{eq:biasdecddproof}
    \end{align}
    where 
    \begin{align*}
        \mathcal{T}_{1, >r} &:= \Projrgg{\sum_{i=1}^{nh}\left(\Delta \check{X}_{ik} - \overline{\Delta \check{X}_{k}}\right)^{\otimes 2}}, \\
        \mathcal{T}_{2,>r} &:= \Projrgg{\sum_{i=1}^{nh}\left(B_{ik}^{\mathrm{cond}} - \overline{B_{k}^{\mathrm{cond}}}\right)^{\otimes 2}}, \\
        \mathcal{T}_{3,>r} &:= \Projrgg{\sum_{i=1}^{nh} \left(B_{ik}^{\mathrm{cond}} - \overline{B_{k}^{\mathrm{cond}}}\right)\left(\Delta \check{X}_{ik} - \overline{\Delta \check{X}_{k}}\right)^\top}.
    \end{align*}
    We first verify the upper bounds stated in \eqref{eq:EWT1Proj}--\eqref{eq:EWT3Proj}. 
    
\noindent \underline{Proof of \eqref{eq:EWT1Proj}.} Using the triangle inequality, we first obtain 
    \begin{align}
        &\bigg\Vert \mathbb{E}\bigg[\Projrgg{\sum_{i=1}^{nh} \left(\Delta \check{X}_{ik} - \overline{\Delta \check{X}_k}\right)^{\otimes 2}}\bigg] \bigg\Vert \leq \left\Vert \EW{\mathcal{T}_{1.1,>r}} \right\Vert + \left\Vert \EW{\mathcal{T}_{1.2,>r}} \right\Vert + 2\left\Vert \EW{\mathcal{T}_{1.3,>r}} \right\Vert, \label{ineq:h0ad1}
    \end{align}
    where 
    \begin{align*}
        \mathcal{T}_{1.1,>r} &:= \Projrgg{\sum_{i=1}^{nh}\left(\Delta \tilde{X}_{ik} - \overline{\Delta \tilde{X}_{k}}\right)^{\otimes 2}},\\
        \mathcal{T}_{1.2, >r} &:= \Projrgg{\sum_{i=1}^{nh} \left(B_{ik} - B_{ik}^{\mathrm{cond}} - \left(\overline{B_k} - \overline{B_{k}^{\mathrm{cond}}}\right)\right)^{\otimes 2}}, \\
        \mathcal{T}_{1.3,>r} &:= \Projrgg{\sum_{i=1}^{nh} \left(\Delta \tilde{X}_{ik} - \overline{\Delta \tilde{X}_k}\right) \left(B_{ik} - B_{ik}^{\mathrm{cond}} - \left(\overline{B_k} - \overline{B_{k}^{\mathrm{cond}}}\right)\right)^{\top}}.
    \end{align*} 
    In the following, we will consider the terms separately. 
    
    \indent We first compute $\Vert \EWt{\mathcal{T}_{1.1, >r}}\Vert$. Since the increments $\Delta \tilde{X}_{ik}$ are independent, it follows that $\EWt{h^{-1} \sum_{i=1}^{nh}(\Delta \tilde{X}_{ik} - \overline{\Delta \tilde{X}_{k}})^{\otimes 2}} = \EWt{h^{-1} \sum_{i=1}^{nh} (\Delta \tilde{X}_{ik})^{\otimes 2} - n (\overline{\Delta \tilde{X}_{k}})^{\otimes 2}} = (1-(nh)^{-1}) \Sigma_{\avg}^{kh}$. Because the projection $\Proj{P}_{>r}$ is deterministic we obtain
    \begin{align}
        \left\Vert \EW{\mathcal{T}_{1.1, >r}} \right\Vert = h\big(1- (nh)^{-1}\big)\big\Vert \Sigma_{\avg, >r}^{kh} \big\Vert. \label{ineq:h0ad2}
    \end{align}
    \indent For $\Vert\EWt{\mathcal{T}_{1.2, >r}}\Vert$, we have $B_{ik} - \overline{B_{k}} = \int_{I_{ik}} b(s) \inte s - (nh)^{-1} \int_{I_{k}} b(t) \inte t = \int_{I_{ik}} (b(s) - h^{-1} \int_{I_{k}} b(t) \inte t) \inte s$ $= h^{-1} \int_{I_{k}} (\int_{I_{ik}} (b(s) - b(t)) \inte s) \inte t$. Applying Jensen's inequality, Fubini's theorem and the H\"older-regularity condition (compare Assumption~\ref{ass:hölderad}) yields 
    \begin{align} \label{ineq:BbarB}
        \EW{\left\Vert B_{ik} - \overline{B_{k}} \right\Vert^2} \leq (nh)^{-1} \int_{I_k}\left(\int_{I_{ik}} \EW{\Vert b(s) - b(t) \Vert^2} \inte s \right) \inte t \leq L_b^2 n^{-2} h^{2\beta_b}.
    \end{align} Combining this with $\Vert \Proj{P}_{>r} \Vert \leq 1$, Lemma~\ref{lem:BcondAvBcondsquared} and the H\"older-regularity condition, it follows that 
    \begin{align}
        \left\Vert\EWt{\mathcal{T}_{1.2, >r}}\right\Vert &\leq\bigg\Vert \mathbb{E}\bigg[\sum_{i=1}^{nh} \left(B_{ik} - B_{ik}^{\mathrm{cond}} - \left(\overline{B_k} - \overline{B_{k}^{\mathrm{cond}}}\right)\right)^{\otimes 2}\bigg] \bigg\Vert \notag \\
        &\leq 2\sum_{i=1}^{nh} \mathbb{E}\bigg[\left\Vert B_{ik} - \overline{B_k} \right\Vert^2 + \left\Vert B_{ik}^{\mathrm{cond}} - \overline{B_{k}^{\mathrm{cond}}} \right\Vert^2\bigg] \notag \\
        &\leq 10 L_b^2 n^{-1} h^{2 \beta_b + 1}. \label{ineq:h0ad3}
    \end{align}
    
    \indent For the third term we get the following, using the Cauchy--Schwarz inequality, 
    \begin{align}
        \left\Vert\EW{\mathcal{T}_{1.3,>r}}\right\Vert
        &\leq \sum_{i=1}^{nh} \bigg( \bigg(\mathbb{E}\bigg[\bigg\Vert\Projrgg{\Delta \tilde{X}_{ik} - \overline{\Delta \tilde{X}_k}}\bigg\Vert^2\bigg]\bigg)^{1/2}  \notag \\
        &\phantom{\leq}\ \qquad\qquad\qquad\qquad\qquad\cdot \left(\EW{\left\Vert B_{ik} - B_{ik}^{\mathrm{cond}} - \left(\overline{B_k} - \overline{B_k^{\mathrm{cond}}}\right)\right\Vert^2}\right)^{1/2}\bigg) \notag \\
        &\leq \sum_{i=1}^{nh} \bigg(\bigg(2\mathbb{E}\bigg[\bigg\Vert\Projrgg{\Delta \tilde{X}_{ik}}\bigg\Vert^2 + \left\Vert\Projr{\overline{\Delta \tilde{X}_k}}\right\Vert^2\bigg]\bigg)^{1/2} \notag \\
        &\phantom{=}\ \qquad\qquad\qquad\qquad\qquad\thickspace\thickspace \cdot  \left(\EW{\left\Vert B_{ik} - B_{ik}^{\mathrm{cond}} - \left(\overline{B_k} - \overline{B_k^{\mathrm{cond}}}\right)\right\Vert^2}\right)^{1/2}\bigg). \label{ineq:h0ad4}
    \end{align}
    Further, we have $\EWt{\Vert B_{ik} - B_{ik}^{\mathrm{cond}} - (\overline{B_k} - \overline{B_{k}^{\mathrm{cond}}})\Vert^2} \leq 2 \EWt{\Vert B_{ik} - \overline{B_k} \Vert^2 + \Vert B_{ik}^{\mathrm{cond}} - \overline{B_{k}^{\mathrm{cond}}} \Vert^2}$. Combining this with Lemma~\ref{lem:BcondAvBcondsquared}, \eqref{ineq:BbarB}, $\Projrt{\Delta \tilde{X}_{ik}} \sim \mathcal{N}(0,\int_{I_{ik}} \Sigma_{>r}(t) \inte t)$ and \eqref{ineq:h0ad4} gives us
    \begin{align}
        &\left\Vert\EW{\mathcal{T}_{1.3,>r}}\right\Vert \notag \\
        &\qquad\leq\sqrt{20} L_b n^{-1}h^{\beta_b}\sum_{i=1}^{nh} \bigg(\bigg(\tr\bigg(\int_{I_{ik}} \Sigma_{>r}(t) \inte t \bigg)\bigg)^{1/2} + (nh)^{-1} \big(\tr_{>r}\big(h\Sigma_{\avg}^{kh}\big)\big)^{1/2}\bigg) \notag \\
        &\qquad\leq \sqrt{20} L_b n^{-1}h^{\beta_b} \big((nh)^{1/2} + 1\big) \big(\tr_{>r}\big(h \Sigma_{\avg}^{kh}\big)\big)^{1/2}. \label{ineq:h0ad5}
    \end{align}
    Plugging the upper bounds \eqref{ineq:h0ad2}, \eqref{ineq:h0ad3} and \eqref{ineq:h0ad5} in \eqref{ineq:h0ad1} and using Assumption~\ref{ass:nh} yields  
    \begin{align*}
        \left\Vert\EW{\mathcal{T}_{1,>r}} \right\Vert
        &\lesssim h\lambda_{r+1}\big(\Sigma_{\avg}^{kh}\big) + L_b^2 n^{-1}h^{2\beta_b+1} + L_bn^{-1/2}h^{\beta_b + 1}\big(\tr_{>r}\big(\Sigma_{\avg}^{kh}\big)\big)^{1/2}.
    \end{align*}
    
\noindent \underline{Proof of \eqref{eq:EWT2Proj}.} 
    For the orthogonal projection $\Proj{P}_{>r}$, we have $\Vert \Proj{P}_{>r} \Vert \leq 1$. This yields
    \begin{align*}
        \left\Vert\EW{\mathcal{T}_{2,>r}} \right\Vert &\leq \mathbb{E}\bigg[\bigg\Vert\sum_{i=1}^{nh} \Big( B_{ik}^{\mathrm{cond}} - \overline{B_k^{\mathrm{cond}}}\Big)^{\otimes 2}\bigg\Vert\bigg] \leq \sum_{i=1}^{nh} \EW{\left\Vert B_{ik}^{\mathrm{cond}} - \overline{B_k^{\mathrm{cond}}}\right\Vert^2} \leq 4 L_b^2 n^{-1} h^{2 \beta_b + 1}, 
    \end{align*}
    where Lemma~\ref{lem:BcondAvBcondsquared} has been used in the last step. 
    
\noindent \underline{Proof of \eqref{eq:EWT3Proj}.} 
    By using the fact that the orthogonal projection $\Proj{P}_{>r}$ is deterministic and satisfies $\Vert\Proj{P}_{>r}\Vert \leq 1$, the linearity of the expectation and the Cauchy--Schwarz inequality, we have
    \begin{align}
        &\bigg\Vert \Projrgg{\mathbb{E}\bigg[\sum_{i=1}^{nh} \left(B_{ik}^{\mathrm{cond}} - \overline{B_{k}^{\mathrm{cond}}}\right)\left(\Delta \check{X}_{ik} - \overline{\Delta \check{X}_{k}}\right)^\top\bigg]}\bigg\Vert \notag \\
        &\qquad\leq \sum_{i=1}^{nh} \bigg(\bigg(\EW{\left\Vert B_{ik}^{\mathrm{cond}} - \overline{B_{k}^{\mathrm{cond}}}\right\Vert^2}\bigg)^{1/2} \bigg(\mathbb{E}\bigg[\left\Vert\Projr{\Delta \check{X}_{ik} -  \overline{\Delta\check{X}_{k}}}\right\Vert^2\bigg]\bigg)^{1/2}\bigg). \label{ineq:h0ad6}
    \intertext{Lemma~\ref{lem:BcondAvBcondsquared} now gives an upper bound for the first factor. Therefore, we will investigate the second factor. For any $i \in \ldbrack 1, nh \rdbrack$, we get }
        &\mathbb{E}\bigg[\left\Vert\Projr{\Delta \check{X}_{ik} -  \overline{\Delta\check{X}_{k}}}\right\Vert^2\bigg] \notag \\
        &\qquad\leq 3 \mathbb{E}\bigg[\left\Vert \Projr{\Delta \tilde{X}_{ik} - \overline{\Delta \tilde{X}_k}}\right\Vert^2 + \left\Vert B_{ik} - \overline{B_{k}} \right\Vert^2 + \left \Vert B_{ik}^{\mathrm{cond}} - \overline{B_k^{\mathrm{cond}}} \right \Vert^2\bigg]. \label{ineq:h0ad7}
    \end{align}
    Since $ \Delta \tilde{X}_{ik} - \overline{\Delta \tilde{X}_{k}} \sim \mathcal{N}(0, (1 - 2 (nh)^{-1})\int_{I_{ik}} \Sigma(t) \inte t + (nh)^{-2} \int_{I_{k}} \Sigma(t) \inte t )$, it follows that $\EWt{\Vert \Delta \tilde{X}_{ik} - \overline{\Delta \tilde{X}_{k}} \Vert^2} \leq \tr((1 - 2 (nh)^{-1})\int_{I_{ik}} \Sigma(t) \inte t + (nh)^{-2} \int_{I_{k}} \Sigma(t) \inte t)$. Furthermore, the projection $\Proj{P}_{>r}$ is deterministic. Combining these with \eqref{ineq:h0ad7}, it follows that 
    \begin{align}
        &\mathbb{E}\bigg[\left\Vert\Projr{\Delta \check{X}_{ik} -  \overline{\Delta\check{X}_{k}}}\right\Vert^2\bigg] \notag \\
        &\qquad\leq 3 \big(1-2(nh)^{-1}\big) \tr\left(\int_{I_{ik}} \Sigma_{>r}(t) \inte t\right) + 3 (nh)^{-2} \tr\left(\int_{I_{k}} \Sigma_{>r}(t) \inte t\right) \notag\\
        &\qquad\phantom{=}+ 3 \EW{\left\Vert B_{ik} - \overline{B_{k}} \right\Vert^2 + \left\Vert B_{ik}^{\mathrm{cond}} - \overline{B_k^{\mathrm{cond}}} \right\Vert^2}. \label{eq:h0hilfseq}
    \end{align}
    Using \eqref{ineq:BbarB} and Lemma~\ref{lem:BcondAvBcondsquared} we can rewrite \eqref{eq:h0hilfseq} as 
    \begin{align}
        &\mathbb{E}\bigg[\left\Vert\Projr{\Delta \check{X}_{ik} -  \overline{\Delta\check{X}_{k}}}\right\Vert^2\bigg] \notag \\
        &\qquad\leq 3 \big(1-2(nh)^{-1}\big)\tr\left(\int_{I_{ik}} \Sigma_{>r}(t) \inte t\right) + 3(nh)^{-2} \tr\left(\int_{I_{k}} \Sigma_{>r}(t) \inte t \right) + 15 L_b^2 n^{-2} h^{2 \beta_b}.\label{ineq:h0ad8}
    \intertext{Combining Lemma~\ref{lem:BcondAvBcondsquared}, \eqref{ineq:h0ad6} and \eqref{ineq:h0ad8} we can assert that}
        &\bigg\Vert \Projrgg{\mathbb{E}\bigg[\sum_{i=1}^{nh} \left(B_{ik}^{\mathrm{cond}} - \overline{B_{k}^{\mathrm{cond}}}\right)\left(\Delta \check{X}_{ik} - \overline{\Delta\check{X}_{k}}\right)^\top}\bigg] \bigg\Vert \notag\\
        &\qquad\leq 2 L_b n^{-1} h^{\beta_b} \sum_{i=1}^{nh}\bigg( 3 \big(1-2(nh)^{-1}\big)\tr\bigg(\int_{I_{ik}} \Sigma_{>r}(t) \inte t\bigg) \notag \\
        &\qquad\qquad\qquad\qquad\quad\phantom{=} + 3(nh)^{-2} \tr\bigg(\int_{I_{k}} \Sigma_{>r}(t) \inte t\bigg) + 15 L_b^2 n^{-2} h^{2 \beta_b}\bigg)^{1/2} \notag \\
        &\qquad\leq 2L_b n^{-1} h^{\beta_b} \sum_{i=1}^{nh}\bigg(\bigg(3 \big(1 - 2(nh)^{-1}\big) \tr\bigg(\int_{I_{ik}} \Sigma_{>r}(t) \inte t\bigg)\bigg)^{1/2} \notag\\
        &\qquad\qquad\qquad\qquad\quad\phantom{=}
         + \bigg(3(nh)^{-2} \tr\bigg(\int_{I_{k}} \Sigma_{>r}(t) \inte t\bigg)\bigg)^{1/2} + \big(15 L_b^2 n^{-2} h^{2 \beta_b}\big)^{1/2}\bigg), \notag
    \intertext{where the concavity of the square root function has been used in the last step. Applying the Cauchy--Schwarz inequality gives us}
        &\bigg\Vert \Projrgg{\mathbb{E}\bigg[\sum_{i=1}^{nh} \left(B_{ik}^{\mathrm{cond}} - \overline{B_{k}^{\mathrm{cond}}}\right)\left(\Delta \check{X}_{ik} - \overline{\Delta\check{X}_{k}}\right)^\top}\bigg] \bigg\Vert \notag\\
        &\qquad\leq 2L_b n^{-1} h^{\beta_b} \bigg((nh)^{1/2}\bigg(3\big(1 - 2(nh)^{-1}\big)\tr\bigg(\int_{I_{k}} \Sigma_{>r}(t) \inte t \bigg)\bigg)^{1/2}\notag\\
        &\qquad\qquad\qquad\qquad\qquad\qquad\qquad\quad\phantom{=} + \bigg(3 \tr\bigg(\int_{I_{k}} \Sigma_{>r}(t) \inte t\bigg)\bigg)^{1/2}  + \sqrt{15} L_bh^{\beta_b+1}\bigg) \notag \\
        &\qquad\lesssim L_b n^{-1/2}h^{\beta_b + 1/2} \tr_{>r}\big(h\Sigma_{\avg}^{kh}\big)^{1/2} + L_b^2 n^{-1}h^{2\beta_b+1}, \notag
    \end{align}
    where Assumption~\ref{ass:nh} has been used in the last step. 
    
\noindent \underline{Proof of the upper bound for the bias.} 
    Combining \eqref{eq:biasdecddproof} with \eqref{eq:EWT1Proj}--\eqref{eq:EWT3Proj}, we get 
    \begin{align*}
        \big\Vert\EWest{\hat{\Sigma}_{>r}^{kh}}\big\Vert
        &\lesssim \lambda_{r+1}\big(\Sigma_{\avg}^{kh}\big) + L_b n^{-1/2}h^{\beta_b} \big(\tr_{>r} \big(\Sigma^{kh}_{\avg}\big)\big)^{1/2} + L_b^2 n^{-1} h^{2\beta_b}. \notag
    \end{align*} 
    Taking $\tr_{>r}(\Sigma_{\avg}^{kh}) \leq (d-r)\lambda_{r+1}(\Sigma_{\avg}^{kh})$ yields 
    \begin{align*}
        \sum_{k=0}^{h^{-1}-1} h \big\Vert \EWest{\hat{\Sigma}_{>r}^{kh}}\big\Vert &\lesssim \bigg(\sum_{k=0}^{h^{-1}-1} h \lambda_{r+1}\big(\Sigma_{\avg}^{kh}\big)\bigg) + L_b^2 n^{-1}h^{2\beta_b}.
    \end{align*}
    
\subsection{Proof of Proposition~\ref{prop:stocherrorh0ad}}\label{app:stocherroroh0add}
    We begin by decomposing $\EWt{\Vert \hat{\Sigma}_{>r}^{kh} - \EWt{\hat{\Sigma}_{>r}^{kh}} \Vert^2}$ into six terms, for which appropriate upper bounds are subsequently derived separately. To obtain these upper bounds, we exploit properties of the Hilbert--Schmidt norm, the trace, and centred multivariate Gaussian increments, together with Jensen's inequality, Fubini's theorem, and the H\"older-regularity of the drift process $b$. The proof is completed by combining the upper bounds  obtained for the individual terms. \\
    \indent Applying the orthogonal projection $\Proj{P}_{>r}$ onto the space $V_{>r}$, generated by the $d-r$ smallest eigenvalues of $\Sigma_{\avg}^{kh}$, to the re-centred covariance estimator implies the following decomposition: 
    \begin{align*}
        &\EWest{\big\Vert \hat{\Sigma}_{>r}^{kh} - \EWest{\hat{\Sigma}_{>r}^{kh}} \big\Vert^2} \\
        &\qquad\leq 2\mathbb{E}\bigg[\bigg\Vert h^{-1}\sum_{i=1}^{nh} \Proj{P}_{>r}\bigg(\left(\Delta X_{ik} - n^{-1}b(kh)\right)^{\otimes 2} - \mathbb{E}\bigg[h^{-1}\sum_{j=1}^{nh} \left(\Delta X_{jk} - n^{-1}b(kh)\right)^{\otimes2}\bigg]\bigg)\bigg|_{V_>r}\bigg\Vert^2\bigg]\\
        &\qquad\, \phantom{=}+ 2\mathbb{E}\bigg[\left\Vert n\Projr{\left(\overline{\Delta X_{k}} - n^{-1}b(kh)\right)^{\otimes2} - \EW{n\left(\overline{\Delta X_{k}} - n^{-1}b(kh)\right)^{\otimes2}}}\right\Vert^2\bigg].
    \end{align*}
    We will consider the two terms separately. Using the triangle inequality, we get
    \begin{align*}
        &\mathbb{E}\bigg[\bigg\Vert h^{-1}\sum_{i=1}^{nh} \Proj{P}_{>r}\bigg(\left(\Delta X_{ik} - n^{-1}b(kh)\right)^{\otimes2} - \mathbb{E}\bigg[h^{-1}\sum_{j=1}^{nh} \left(\Delta X_{jk} - n^{-1}b(kh)\right)^{\otimes2}\bigg]\bigg)\bigg|_{V_>r}\bigg\Vert^2\bigg] \\
        &\qquad\leq 3 \left(\EW{\left\Vert \mathcal{T}_{4,>r} \right\Vert^2} + \EW{\left\Vert \mathcal{T}_{5,>r}\right\Vert^2} + 2 \EW{\left\Vert \mathcal{T}_{6, >r}\right\Vert^2}\right),
    \end{align*}
    where $\mathcal{T}_{4,>r}$, $\mathcal{T}_{5,>r}$ and $\mathcal{T}_{6,>r}$ are defined in \eqref{eq:T4Proj}--\eqref{eq:T6Proj}, respectively. 
    
    For the second term, we obtain, by applying the triangle inequality,
    \begin{align*}
        &\mathbb{E}\bigg[\bigg\Vert n\Projr{\left(\overline{\Delta X_{k}} - n^{-1}b(kh)\right)^{\otimes2} - \EW{n\left(\overline{\Delta X_{k}} - n^{-1}b(kh)\right)^{\otimes2}}}\bigg\Vert^2\bigg] \\
        &\qquad\leq 3n^{-2} h^{-4} \left(\EW{\left\Vert \mathcal{T}_{7,>r}\right\Vert^2} + \EW{\left\Vert \mathcal{T}_{8, >r} \right\Vert^2} + 2\EW{\left\Vert \mathcal{T}_{9, >r}\right\Vert^2}\right),
    \end{align*}
    with $\mathcal{T}_{7,>r}$, $\mathcal{T}_{8,>r}$ and $\mathcal{T}_{9,>r}$ defined in \eqref{eq:T7Proj}--\eqref{eq:T9Proj}, respectively. 
    
    \noindent \underline{Proof of \eqref{eq:EWT4Proj}.} We apply $\Vert A \Vert \leq \HS{A}$ for any $A \in \mathbb{R}^{d \times d}$ and use the linearity of the Hilbert--Schmidt norm. Then Lemma~\ref{lem:HSEukl} with $\Projrt{\Delta \tilde{X}_{ik}} \sim \mathcal{N}(0, \int_{I_{ik}}\Sigma_{>r}(t) \inte t)$ yields
    \begin{align*}
        &\mathbb{E}\bigg[\bigg\Vert h^{-1}\sum_{i=1}^{nh} \bigg(\Projro{\big(\Delta \tilde{X}_{ik}\big)^{\otimes 2}} - \EW{\Projro{\big(\Delta \tilde{X}_{ik}\big)^{\otimes2}}}\bigg)\bigg\Vert^2\bigg] \\
        &\qquad\leq h^{-2} \bigg(\mathbb{E}\bigg[\HSgg{\sum_{i=1}^{nh} \bigg(\Projro{\big(\Delta \tilde{X}_{ik}\big)}\bigg)^{\otimes 2}}^2\bigg] - \HSgg{\sum_{i=1}^{nh}\mathbb{E}\bigg[\bigg(\Projro{\big(\Delta \tilde{X}_{ik}\big)}\bigg)^{\otimes2}}\bigg]^2\bigg) \\ 
        &\qquad\leq 2h^{-2} \sum_{i=1}^{nh} \left(\tr\left(\int_{I_{ik}} \Sigma_{>r}(t) \inte t \right)\right)^2 \\
        &\qquad\leq 2 \big(\tr_{>r}\big(\Sigma_{\avg}^{kh}\big)\big)^2, 
    \end{align*}
    where $\sum_{i=1}^n a_i^2 \leq (\sum_{i=1}^n a_i)^2$ for $a_i \geq 0$, $i \in \ldbrack 1, n\rdbrack$, has been used in the last step. 
    
    \noindent \underline{Proof of \eqref{eq:EWT5Proj}.} Because of $\Vert \Proj{P}_{>r} \Vert \leq 1$ and the triangle inequality, we get 
    \begin{align*}
        \EW{\Vert \mathcal{T}_{5, >r} \Vert^2} &\leq h^{-2} \mathbb{E}\bigg[\bigg(\sum_{i=1}^{nh} \bigg\Vert \int_{I_{ik}} \left(b(t) - b(kh)\right) \inte t \bigg\Vert^2\bigg)^2\bigg] \\
        &\leq n^{-2}h^{-1} \int_{I_k} \EW{\left\Vert b(t) - b(kh) \right\Vert^4} \inte t,
    \intertext{where we used Jensen's inequality and Fubini's theorem in the last step. Applying the H\"older-regularity condition (compare Assumption~\ref{ass:hölderad}) gives us  }
        \EW{\Vert \mathcal{T}_{5, >r} \Vert^2} &\leq L_b^4 n^{-2} h^{4 \beta_b}. 
    \end{align*}
    \noindent \underline{Proof of \eqref{eq:EWT6Proj}.} From Corollary~\ref{cor:RVwithNormalMultiplikation} with $V_i = \Projrt{\int_{I_{ik}} \left(b(t) - b(kh)\right) \inte t}$ and $W_i = \Projrt{\Delta \tilde{X}_{ik}} \sim \mathcal{N}(0, \int_{I_{ik}} \Sigma_{>r}(t) \inte t)$ and using $\HSt{A} \leq \tr(A)$ for any $A \in \mathbb{R}^{d \times d}_{\spd}$ as well as $\Vert \Proj{P}_{>r} \Vert \leq 1$, we get 
    \begin{align} 
        &\EW{\left\Vert\mathcal{T}_{6, >r}\right\Vert^2} \notag \\
        &\qquad\leq \sqrt{3}nh^{-1} \sum_{i=1}^{nh} \bigg(\mathbb{E}\bigg[\bigg \Vert \int_{I_{ik}} \left(b(t) - b(kh)\right) \inte t \bigg \Vert^4\bigg] \bigg(\tr\bigg(\int_{I_{ik}} \Sigma_{>r}(t) \inte t\bigg)\bigg)^2\bigg)^{1/2}. \label{eq:T6UB}
    \end{align}
    Applying Jensen's inequality, Fubini's theorem and the H\"older-regularity condition (compare Assumption~\ref{ass:hölderad}), we can bound $\EWt{\Vert \int_{I_{ik}} \left(b(t) - b(kh)\right) \inte t \Vert^4} \leq n^{-3} \EWt{\int_{I_{ik}} \Vert b(t) - b(kh) \Vert^4 \inte t} \leq L_b^4 n^{-4} h^{4 \beta_b}$. Combining this with \eqref{eq:T6UB} we obtain 
        $$\EW{\left\Vert\mathcal{T}_{6, >r}\right\Vert^2} 
        \leq \sqrt{3}L_b^2 n^{-1} h^{2 \beta_b} \tr_{>r}\big( \Sigma_{\avg}^{kh}\big), $$
    where we used the linearity of the trace. 
    
    \noindent \underline{Proof of \eqref{eq:EWT7Proj}.} Analogously to the proof of \eqref{eq:EWT4Proj}, we have 
    $$
        \EW{\Vert \mathcal{T}_{7, >r} \Vert^2} \leq \mathbb{E}\bigg[\HSgg{\left(\Projr{ nh \overline{\Delta \tilde{X}_{k}}}\right)^{\otimes 2}}^2\bigg] - \HSgg{\mathbb{E}\bigg[\bigg(\Projr{ nh \overline{\Delta \tilde{X}_{k}}}\bigg)^{\otimes2}}\bigg]^2
    $$
    and hence, by using Lemma~\ref{lem:HSEukl} with $\Projrt{nh \overline{\Delta \tilde{X}_k}} \sim \mathcal{N}(0,\int_{I_k} \Sigma_{>r}(t) \inte t)$, 
    $$
        \EW{\Vert \mathcal{T}_{7, >r} \Vert^2} \leq 2h^2 \big(\tr_{>r}\big(\Sigma_{\avg}^{kh}\big)\big)^2. 
    $$
    \noindent \underline{Proof of \eqref{eq:EWT8Proj}.} As in the proof of \eqref{eq:EWT5Proj}, we get, by applying Jensen's inequality, Fubini's theorem and the H\"older-regularity condition (compare Assumption~\ref{ass:hölderad}), 
    $$
        \EW{\left\Vert \mathcal{T}_{8, >r} \right\Vert} \leq h^3 \int_{I_k} \EW{\left\Vert b(t) - b(kh) \right\Vert^4} \inte t \leq L_b^4 h^{4(\beta_b +1)}. 
    $$
    \noindent \underline{Proof of \eqref{eq:EWT9Proj}.} We have  
    \begin{align*}
        \EW{\left \Vert \mathcal{T}_{9, >r} \right\Vert^2} &\leq \mathbb{E}\bigg[\bigg\Vert \Projrgg{\bigg(\int_{I_{k}} (b(t) - b(kh)) \inte t \bigg) \left(nh \overline{\Delta \tilde{X}_{k}} \right)^{\top}}\bigg\Vert^2\bigg]. 
    \intertext{By $\Vert \Proj{P}_{>r} \Vert \leq 1$ and the Cauchy--Schwarz inequality, we obtain }
        \mathbb{E}\bigg[\bigg \Vert \mathcal{T}_{9, >r} \bigg\Vert^2\bigg] &\leq \bigg(\mathbb{E}\bigg[\bigg\Vert \int_{I_k} \left(b(t) - b(kh)\right) \inte t \bigg\Vert^4\bigg] \mathbb{E}\bigg[\bigg\Vert \Projr{nh \overline{\Delta \tilde{X}_{k}}} \bigg \Vert^4\bigg] \bigg)^{1/2} \\
        &\leq \sqrt{3} L_b^2 h^{2(\beta_b + 1)} \tr_{>r}\big(h \Sigma^{kh}_{\avg}\big),
    \end{align*}
    where the last step follows from Jensen's inequality, Fubini's theorem, the Hölder regularity condition (cf.\ Assumption~\ref{ass:hölderad}), and the identity $\EWt{\Vert X \Vert^4} = \left(\tr\left(\Sigma\right)\right)^2 + 2\HS{\Sigma}^2$ valid for any 
    $d$-dimensional Gaussian random vector $X \sim \mathcal{N}(0,\Sigma)$ with covariance matrix $\Sigma \in \mathbb{R}^{d \times d}$, together with the bound $\HSt{\Sigma} \leq \tr(\Sigma)$ for all $\Sigma \in \mathbb{R}^{d \times d}_{\spd}$. 
    
    \noindent \underline{Proof of the upper bound for the stochastic error.} Combining the upper bounds on \eqref{eq:EWT4Proj}--\eqref{eq:EWT9Proj} we find 
    \begin{align*}
        &\EWest{\big\Vert \big(\hat{\Sigma}_{>r}^{kh} - \EWest{\hat{\Sigma}_{>r}^{kh}} \big)\big\Vert^2}^{1/2} \\
        &\qquad\lesssim \left(\big(\tr_{>r}\big(\Sigma_{\avg}^{kh}\big)\big)^2 + L_b^4 n^{-2}h^{4 \beta_b} + L_b^2 n^{-1} h^{2\beta_b} \tr_{>r}\big(\Sigma_{\avg}^{kh}\big)\right)^{1/2} \\
         &\qquad\phantom{=}+ \left(n^{-2}h^{-4} \left(h^2 \big(\tr_{>r} \big(\Sigma_{\avg}^{kh}\big)\big)^2 + L_b^4 h^{4 (\beta_b +1)} + L_b^2 h^{2 \beta_b + 3} \tr_{>r}\big(\Sigma_{\avg}^{kh}\big)\right)\right)^{1/2},
    \end{align*}
    and by applying the Binomial theorem, it follows that 
    \begin{align*}
       \EWest{\big\Vert \hat{\Sigma}_{>r}^{kh} - \EWest{\hat{\Sigma}_{>r}^{kh}}\big\Vert^2}^{1/2} 
        &\lesssim \big(1 + (nh)^{-1}\big) \tr_{>r}\big(\Sigma_{\avg}^{kh}\big) + L_b^2 n^{-1}h^{2 \beta_b}. 
    \end{align*}
    Using Assumption~\ref{ass:nh} and $\tr_{>r}(A) \leq (d-r) \lambda_{r+1}(A)$ for any matrix $A \in \mathbb{R}^{d \times d}$, yields
    \begin{align*}
        \EWest{\big\Vert \hat{\Sigma}_{>r}^{kh} - \EWest{\hat{\Sigma}_{>r}^{kh}}\big\Vert^2}^{1/2} 
        \lesssim \lambda_{r+1}\big(\Sigma_{\avg}^{kh}\big) + L_b^2 n^{-1}h^{2\beta_b}. 
    \end{align*}
    Therefore, we obtain 
    \begin{align*}
        \sum_{k=0}^{h^{-1}-1} h \EWest{\big\Vert \hat{\Sigma}_{>r}^{kh} - \EWest{\hat{\Sigma}_{>r}^{kh}}\big\Vert} &\leq \sum_{k=0}^{h^{-1}-1} h \left(\EWest{\big\Vert \hat{\Sigma}_{>r}^{kh} - \EWest{\hat{\Sigma}_{>r}^{kh}} \big\Vert^2}\right)^{1/2} \\
        &\lesssim \bigg(\sum_{k=0}^{h^{-1}-1} h \lambda_{r+1}\big(\Sigma_{\avg}^{kh}\big)\bigg) + L_b^2n^{-1}h^{2\beta_b}, 
    \end{align*}
    which completes the proof. 
\qed

\section{Auxiliary Results for the Proof of Theorem~\ref{thm:h1}} \label{a:alternative}
\begin{lemma} \label{lem:decomph1}
    Let $(\Sigma,b) \in \mathcal{H}_1(r, \beta_b, L_b, h_n, v, C)$, where $r \in \ldbrack 0, d-1\rdbrack$, $\beta_b \in(0,1]$, $L_b \geq 0$, $(h_n)_{n \in \mathbb{N}} \subset (0,1)$ and $v,C > 0$. Then we have
    \begin{align}
        \PPind{(\Sigma,b)}{\varphi_{\alpha, n} = 1} 
        &\geq 1 - \bigg[\mathbb{P}_{(\Sigma,b)}\bigg(\sum_{k=0}^{h_n^{-1}-1} \lambda_{r+1}\bigg(\sum_{i=1}^{nh_n} \big(\Delta \tilde{X}_{ik}\big)^{\otimes 2}\bigg) \leq \kappa_{\alpha, n} + \tilde{\kappa}_{\alpha, n}\bigg) \notag \\
        &\qquad \qquad \qquad \qquad \qquad \qquad \qquad + \mathbb{P}_{(\Sigma,b)}\bigg(\sum_{k=0}^{h_n^{-1}-1} h_n \big\Vert \mathcal{R}^{kh_n} \big\Vert \geq \tilde{\kappa}_{\alpha, n}\bigg)\bigg], \notag 
    \end{align}
    where $\kappa_{\alpha, n}, \tilde{\kappa}_{\alpha,n}>0$, and for any $k \in \ldbrack 0, h_n^{-1}-1 \rdbrack$ we define
    $$\mathcal{R}^{kh_n} := n\left(\overline{\Delta \tilde{X}_{k}}\right)^{\otimes 2} - h_n^{-1} \sum_{i=1}^{nh_n}\left(\Delta \tilde{X}_{ik} \odot \left(B_{ik}- \overline{B_k}\right)\right). $$
\end{lemma}
\begin{proof}
    The proof relies primarily on Weyl's inequality and Bonferroni's inequality. For details, see Appendix~\ref{proof:decomph1}.
\end{proof}
We now derive an upper bound for the first probability appearing in the previous lemma. The constants $C_{d, r+1} > 1$ are chosen as in Corollary S.9 of \cite{ReissWinkelmann2023}.
\begin{lemma}\label{lem:alternativeMR}
    Consider for sample size $n$ the test 
    $$
        \tilde{\varphi}_{\alpha,n} := \indikator \bigg(\sum_{k=0}^{h_n^{-1}-1} \lambda_{r+1}\bigg( \sum_{i=1}^{nh_n} \big(\Delta \tilde{X}_{ik}\big)^{\otimes 2}\bigg) \geq \check{\kappa}_{\alpha,n} \bigg)
    $$ 
    with block sizes $h_n$ and some critical values $\check{\kappa}_{\alpha,n}$. Assume $h_n \rightarrow 0$ as $n \rightarrow \infty$ and that the number of increments per block satisfies $nh_n \geq 2(r+1)C_{d,r+1}$ with a fixed constant $C_{d,r+1} > 1$ only depending on $d$ and $r+1$. Then, $\tilde{\varphi}_{\alpha,n}$ is uniformly consistent over the local alternatives $\mathcal{H}_1(r, \beta_b, L_b, \hbar_n, v_n, C)$, provided $\hbar_n/h_n \rightarrow \infty$, and the rate $v_n$ is larger than a constant multiple of $\check{\kappa}_{\alpha,n}$: 
    $$
        \exists \delta > 0: \lim_{n \rightarrow \infty} \ \inf \left\{ \PPind{(\Sigma,b)}{\tilde{\varphi}_{\alpha,n} = 1} \left| \thinspace (\Sigma,b) \in \mathcal{H}_1(r, \beta_b, L_b, \hbar_n, \delta \check{\kappa}_{\alpha,n}, C) \right. \right\} = 1.
    $$
\end{lemma}
\begin{proof}
    For a proof, we refer to \cite[Theorem 3.5]{ReissWinkelmann2023}, where the alternative must be adjusted accordingly. Moreover, the estimator considered there corresponds in our setting to $h_n^{-1} \sum_{i=1}^{nh_n} (\Delta \tilde{X}_{ik})^{\otimes 2}$ with $k \in \ldbrack0, h_n^{-1}-1\rdbrack$. The proof relies on a Laplace transform result, which yields to a deviation inequality for a triangular scheme of Wishart matrices, and subsequently applies the classical Chernoff argument. 
\end{proof}
Next, we derive an upper bound for the second probability appearing in Lemma~\ref{lem:decomph1}. 
\begin{lemma}\label{lem:alternative2ndterm}
    Assume that $b$ satisfies Assumption~\ref{ass:hölderad} and suppose that the sequence
    $(h_n)_{n \in \mathbb{N}} \subset (0,1) $ tends to zero as $n \rightarrow \infty$. Let $(c_{n})_{n \in \mathbb{N}}$ be a sequence such that $c_{n} \rightarrow \infty$ for $n \rightarrow \infty$. Further, 
    let $(\Sigma,b) \in \mathcal{H}_1(r, \beta_b, L_b, h_{n}, v, C)$, $r \in \ldbrack 0, d-1\rdbrack$, $\beta_b \in(0,1]$, $L_b \geq 0$ and $v, C > 0$. Set $\overline{\Sigma} := \max_{k \in \ldbrack 0, h_n^{-1}-1 \rdbrack} \tr(\int_{I_k} \Sigma(t) \inte t)$ and 
    $$
        \tilde{\kappa}_{\alpha,n}(\Sigma) := c_{n}\big(\overline{\Sigma} n^{-1}h_{n}^{-2} + 2 L_b \overline{\Sigma}^{1/2}n^{-1/2} h_n^{\beta_b-1/2} \big).
    $$ 
    We get 
    $$ \sup_{(\Sigma, b) \in \mathcal{H}_1(r, \beta_b, L_b, h_n, v, C)}\mathbb{P}_{(\Sigma,b)}\bigg(\sum_{k=0}^{h_n^{-1}-1} h_n \big\Vert \mathcal{R}^{kh_n} \big\Vert \geq \tilde{\kappa}_{\alpha,n}(\Sigma)\bigg) \rightarrow 0$$
    for $n \rightarrow \infty$.
\end{lemma}
\begin{proof}
    The proof is based on an application of Markov's inequality. We first decompose the corresponding expectation. To derive suitable upper bounds for the terms arising from the decomposition of the corresponding expectation, we exploit the H\"older-regularity of the process term $b$, Jensen's inequality, and Fubini's theorem, together with standard properties of centred multivariate Gaussian increments and the trace. See Appendix~\ref{proof:alternative2ndterm} for a detailed proof. 
\end{proof}
\subsection{Proof of Lemma~\ref{lem:decomph1}} \label{proof:decomph1}
    Let $k \in \ldbrack 0, h_n^{-1}-1 \rdbrack$. We have the following decomposition
    $$
        \hat{\Sigma}^{kh_n} = h_n^{-1} \sum_{i=1}^{nh_n}\big(\Delta \tilde{X}_{ik}\big)^{\otimes 2} - \tilde{\mathcal{R}}^{kh_n},
    $$
    where 
    \begin{align*}
        \tilde{\mathcal{R}}^{kh_n} &:= n\left(\overline{\Delta \tilde{X}_{k}}\right)^{\otimes 2} - h_n^{-1} \sum_{i=1}^{nh_n} \left(B_{ik} - \overline{B_k}\right)^{\otimes 2} - h_n^{-1} \sum_{i=1}^{nh_n}\left(\Delta \tilde{X}_{ik} \odot \left(B_{ik}- \overline{B_k}\right)\right).
    \end{align*}
    Since $h_n^{-1} \sum_{i=1}^{nh_n} (B_{ik} - \overline{B_k})^{\otimes 2} \succcurlyeq 0$, it follows that
    $$
        \hat{\Sigma}^{kh_n} \succcurlyeq h_n^{-1} \sum_{i=1}^{nh_n}\big(\Delta \tilde{X}_{ik}\big)^{\otimes 2} - \mathcal{R}^{kh_n},
    $$
    where $\mathcal{R}^{kh_n} := \tilde{\mathcal{R}}^{kh_n} + h_n^{-1} \sum_{i=1}^{nh_n} (B_{ik} - \overline{B_k})^{\otimes 2}$. 
    Using Weyl's inequality \citep[see e.g.][Theorem 3.2.1]{BhatiaMatrixAnalysis}, we obtain 
    $$
        \lambda_{r+1}\big(\hat{\Sigma}^{kh_n}\big) \geq \lambda_{r+1} \bigg(h_n^{-1} \sum_{i=1}^{nh_n} \big(\Delta \tilde{X}_{ik}\big)^{\otimes 2} - \mathcal{R}^{kh_n}\bigg) \geq \lambda_{r+1}\bigg(h_n^{-1} \sum_{i=1}^{nh_n} \big(\Delta \tilde{X}_{ik}\big)^{\otimes 2}\bigg) - \left\Vert \mathcal{R}^{kh_n} \right\Vert.
    $$
    This yields, by applying Bonferroni's inequality,  
    \begin{align}
        &\PPind{(\Sigma,b)}{\varphi_{\alpha, n} = 0} = \mathbb{P}_{(\Sigma,b)}\bigg(\sum_{k=0}^{h_n^{-1}-1} h_n \lambda_{r+1}\big(\hat{\Sigma}^{kh_n}\big) \leq \kappa_{\alpha,n}\bigg) \notag \\
        &\qquad\leq \mathbb{P}_{(\Sigma,b)}\bigg(\sum_{k=0}^{h_n^{-1}-1} h_n \bigg(\lambda_{r+1}\bigg(h_n^{-1} \sum_{i=1}^{nh_n} \big(\Delta \tilde{X}_{ik}\big)^{\otimes 2}\bigg) - \left\Vert\mathcal{R}^{kh_n} \right\Vert\bigg) \leq \kappa_{\alpha,n}\bigg) \notag\\
        &\qquad\leq \mathbb{P}_{(\Sigma,b)}\bigg(\sum_{k=0}^{h_n^{-1}-1} h_n \lambda_{r+1}\bigg(h_n^{-1} \sum_{i=1}^{nh_n} \big(\Delta \tilde{X}_{ik}\big)^{\otimes 2}\bigg) \leq \kappa_{\alpha,n} + \tilde{\kappa}_{\alpha,n}\bigg) + \mathbb{P}_{(\Sigma,b)}\bigg(\sum_{k=0}^{h_n^{-1}-1} h_n \left\Vert \mathcal{R}^{kh_n} \right\Vert \geq \tilde{\kappa}_{\alpha,n}\bigg), \notag
    \end{align}
    where $\kappa_{\alpha,n}, \tilde{\kappa}_{\alpha,n} > 0$, and we conclude 
    \begin{align*}
        \PPind{(\Sigma,b)}{\varphi_{\alpha,n} = 1}
        &\geq 1 - \bigg[\mathbb{P}_{(\Sigma,b)}\bigg(\sum_{k=0}^{h_n^{-1}-1} h_n \lambda_{r+1}\bigg(h_n^{-1} \sum_{i=1}^{nh_n} \big(\Delta \tilde{X}_{ik}\big)^{\otimes 2}\bigg) \leq \kappa_{\alpha, n} + \tilde{\kappa}_{\alpha, n}\bigg) \\
        &\quad\qquad\qquad\qquad\qquad\qquad\qquad\qquad\qquad+ \mathbb{P}_{(\Sigma,b)}\bigg(\sum_{k=0}^{h_n^{-1}-1} h_n \left\Vert \mathcal{R}^{kh_n} \right\Vert \geq \tilde{\kappa}_{\alpha, n}\bigg)\bigg],
    \end{align*}
    which completes the proof. 
\qed

\subsection{Proof of Lemma~\ref{lem:alternative2ndterm}} \label{proof:alternative2ndterm}
    Let $(\Sigma,b) \in \mathcal{H}_1(r, \beta_b, L_b, \hbar_{n}, v, C) =: \mathcal{H}_1$, $r \in \ldbrack 0, d-1\rdbrack$, $\beta_b \in(0,1]$, $L_b \geq 0$, $(\hbar_n)_{n \in \mathbb{N}} \subset (0,1)$ and $v, C > 0$. Using the triangle inequality and the norm properties, we get 
    \begin{align}
        &\sup_{(\Sigma,b) \in \mathcal{H}_1}\mathbb{P}_{(\Sigma,b)}\bigg(\sum_{k=0}^{h_n^{-1}-1} h_n \left\Vert \mathcal{R}^{kh_n} \right\Vert \geq \tilde{\kappa}_{\alpha,n}(\Sigma)\bigg) \notag \\
        &\qquad\leq \sup_{(\Sigma,b) \in \mathcal{H}_1}\mathbb{P}_{(\Sigma,b)}\bigg(\sum_{k=0}^{h_n^{-1} - 1}\bigg[(nh_n)^{-1}  \bigg\Vert \sum_{i=1}^{nh_n} \Delta \tilde{X}_{ik} \bigg\Vert^2 + 2 \sum_{i=1}^{nh_n} \big\Vert \Delta \tilde{X}_{ik} \big\Vert  \left\Vert B_{ik} - \overline{B_k}\right\Vert \bigg]\geq \tilde{\kappa}_{\alpha,n}(\Sigma)\bigg). \label{eq:h1absch1}
    \end{align}
    \indent To apply Markov's inequality, we first compute the expectations separately. Let $i \in \ldbrack 1, nh_n\rdbrack$ and $k \in \ldbrack 0, h_n^{-1}-1 \rdbrack$. Since $\Delta \tilde{X}_{ik} \sim \mathcal{N}(0,\int_{I_{ik}}\Sigma(t) \inte t)$, and hence $\sum_{i=1}^{nh_n} \Delta \tilde{X}_{ik} \sim \mathcal{N}(0,\int_{I_{k}}\Sigma(t) \inte t)$, we have 
    \begin{align}
        (nh_n)^{-1} \sum_{k=0}^{h_n^{-1}-1} \mathbb{E}_{(\Sigma, b)}\bigg[\bigg\Vert \sum_{i=1}^{nh_n} \Delta \tilde{X}_{ik}\bigg\Vert^2\bigg] &= (nh_n)^{-1} \sum_{k=0}^{h_n^{-1}-1} \tr\bigg(\int_{I_{k}} \Sigma(t) \inte t \bigg) \notag \\
        &\leq n^{-1}h_n^{-2} \max_{k \in \ldbrack 0, h_n^{-1}-1\rdbrack}\tr\bigg( \int_{I_k} \Sigma(t) \inte t \bigg). \label{eq:h1absch2}
    \end{align}
    \indent Next, we observe, by applying Jensen's inequality, Fubini's theorem and the H\"older-regularity condition (compare Assumption~\ref{ass:hölderad}), 
    \begin{align}
        &\EWind{(\Sigma, b)}{ \left\Vert B_{ik} - \overline{B_k} \right\Vert^2} \notag \\
        &\qquad \leq (nh_n)^{-1} \int_{I_{ik}} \left(\int_{I_{k}} \mathbb{E}_{(\Sigma, b)}\left[\left\Vert b(s) - b(t) \right\Vert^2\right] \inte t \right) \inte s \leq 
        L_b^2 n^{-2} h_n^{2 \beta_b}. \label{eq:h1absch3}
    \end{align}
    \indent For the second summand, we obtain, using the linearity of the expectation and the Cauchy--Schwarz inequality, 
    \begin{align}
        &\mathbb{E}_{(\Sigma, b)}\bigg[ \sum_{k=0}^{h_n^{-1}-1} \sum_{i=1}^{nh_n} \big\Vert \Delta \tilde{X}_{ik} \big\Vert \left\Vert B_{ik} - \overline{B_{k}} \right\Vert\bigg] \notag \\
        &\qquad\leq \sum_{k=0}^{h_n^{-1}-1} \sum_{i=1}^{nh_n} \left(\EWind{(\Sigma, b)}{ \big\Vert \Delta \tilde{X}_{ik} \big\Vert^2}\right)^{1/2} \left(\EWind{(\Sigma, b)}{\left\Vert B_{ik} - \overline{B_{k}} \right\Vert^2}\right)^{1/2}. \label{eq:h1absch4}
    \end{align}
    Combining \eqref{eq:h1absch3} and \eqref{eq:h1absch4}, and using $\Delta \tilde{X}_{ik} \sim \mathcal{N}(0,\int_{I_{ik}} \Sigma(t) \inte t)$ yields
    \begin{align}
        &\mathbb{E}_{(\Sigma, b)}\bigg[2 \sum_{k=0}^{h_n^{-1}-1} \sum_{i=1}^{nh_n} \big\Vert \Delta \tilde{X}_{ik} \big\Vert \left\Vert B_{ik} - \overline{B_{k}} \right\Vert\bigg] \notag \\
        &\qquad\leq 2 L_b n^{-1} h_n^{\beta_b} \sum_{k=0}^{h_n^{-1}-1} \sum_{i=1}^{nh_n} \bigg( \tr\bigg( \int_{I_{ik}} \Sigma(t) \inte t \bigg)\bigg)^{1/2}. \notag 
    \end{align}
    Jensen's inequality now leads to
    \begin{align}
        &\mathbb{E}_{(\Sigma, b)}\bigg[2 \sum_{k=0}^{h_n^{-1}-1} \sum_{i=1}^{nh_n} \big\Vert \Delta \tilde{X}_{ik} \big\Vert \left\Vert B_{ik} - \overline{B_k} \right\Vert\bigg] \notag \\  
        &\qquad\leq 2 L_b n^{-1/2} h_n^{\beta_b + 1/2} \sum_{k=0}^{h_n^{-1}-1} \bigg(\tr\bigg(\int_{I_k} \Sigma(t) \inte t \bigg)\bigg)^{1/2} \notag \\
        &\qquad\leq 2 L_b n^{-1/2} h_n^{\beta_b - 1/2} \bigg(\max_{k \in \ldbrack 0, h_n^{-1}-1\rdbrack} \tr\bigg(\int_{I_k} \Sigma(t) \inte t \bigg)\bigg)^{1/2}. \label{eq:h1absch5}
    \end{align}
    \indent Defining $\overline{\Sigma} := \max_{k \in \ldbrack 0, h_n^{-1}-1 \rdbrack} \tr(\int_{I_k} \Sigma(t) \inte t)$, using Markov's inequality and combining \eqref{eq:h1absch1}, \eqref{eq:h1absch2} and \eqref{eq:h1absch5}, we get
    \begin{align}
        &\sup_{(\Sigma,b) \in \mathcal{H}_1} \mathbb{P}_{(\Sigma, b)}\bigg(\sum_{k=0}^{h_n^{-1}-1} h_n \big\Vert \mathcal{R}^{kh_n} \big\Vert \geq \tilde{\kappa}_{\alpha,n}(\Sigma)\bigg) \notag\\
        &\qquad\leq \sup_{(\Sigma,b) \in \mathcal{H}_1} (\tilde{\kappa}_{\alpha,n}(\Sigma))^{-1} \big( \overline{\Sigma} n^{-1}h_n^{-2} + 2 L_b \overline{\Sigma}^{1/2} n^{-1/2} h_n^{\beta_b - 1/2} \big). \notag
    \end{align}
    \indent Since we have $\tilde{\kappa}_{\alpha,n}(\Sigma) = c_{n}(\overline{\Sigma}  n^{-1}h_{n}^{-2} + 2 L_b \overline{\Sigma}^{1/2} n^{-1/2} h_n^{\beta_b - 1/2} )$ and $c_{n} \rightarrow \infty$ with $n \rightarrow \infty$, we conclude that 
    $$
        \sup_{(\Sigma, b) \in \mathcal{H}_1}\mathbb{P}_{(\Sigma, b)}\bigg(\sum_{k=0}^{h_n^{-1}-1} h_n \big\Vert \mathcal{R}^{kh_n} \big\Vert \geq \tilde{\kappa}_{\alpha,n}(\Sigma)\bigg) \rightarrow 0
    $$
    for $n \rightarrow \infty$.
\qed

\section{Useful inequalities}
\begin{lemma}\label{lem:NormDiffXiMean}
    Let $(\Omega, \mathcal{F}, \mathbb{P})$ be a probability space, and let $X_i: \Omega \rightarrow \mathbb{R}^{m}$, $i \in \ldbrack 1,n \rdbrack$, be random vectors such that $\EWt{\Vert X_i \Vert^2} < \infty$. Then we have 
    $$
        \mathbb{E}\bigg[\bigg\Vert X_i - n^{-1} \sum_{j=1}^n X_j \bigg\Vert^2\bigg] \leq 2 \bigg(\EW{\left\Vert X_i \right\Vert^2} + n^{-1} \sum_{j=1}^n \EW{\left\Vert X_j \right\Vert^2}\bigg).
    $$
\end{lemma} 
\begin{proof}
    From $(a+b)^2 \leq 2 (a^2 + b^2)$ for any $a,b \in \mathbb{R}$ and Jensen's inequality the claim follows. 
\end{proof}

\begin{lemma}\label{lem:HSEukl}
    Let $X_i \sim \mathcal{N}(0,\Sigma_i)$, $i \in\ldbrack 1, n \rdbrack$, be independent $d$-dimensional Gaussian random vectors. Define $X := (X_1^\top,\ldots, X_n^\top)^\top \in \mathbb{R}^{n \times d}$. Then we have 
    $$
        \EW{\big\Vert X^\top X\big\Vert_{\mathrm{HS}}^2} - \big\Vert\mathbb{E}\big[X^\top X\big]\big\Vert_{\mathrm{HS}}^2 \leq 2 \sum_{i=1}^n \left(\tr\left(\Sigma_i \right)\right)^2.
    $$
\end{lemma}
\begin{proof}
    For any $i \in \ldbrack 1, n\rdbrack$, the fourth-order central moments of the multivariate normal distribution yield
    \begin{align*}
        \EW{\big\Vert X_i X_i^\top\big\Vert_{\mathrm{HS}}^2} = \EW{\Vert X_i \Vert^4} = \left(\tr\left(\Sigma_i\right)\right)^2 + 2\HS{\Sigma_i}^2.
    \end{align*}
    Therefore, we obtain
    \begin{align*}
        \EW{\big\Vert X^\top X\big\Vert_{\mathrm{HS}}^2} - \big\Vert\mathbb{E}\big[X^\top X\big]\big\Vert_{\mathrm{HS}}^2
        = \sum_{i=1}^{n} \left(\HS{\Sigma_{i}}^2 + \left(\tr\left(\Sigma_{i}\right)\right)^2\right)
        \leq 2 \sum_{i=1}^{n} (\tr(\Sigma_i))^2, 
    \end{align*}
    where in the last step is the fact used that $\HSt{A}^2 \leq (\tr(A))^2$ for any $A \in \mathbb{R}_{\spd}^{d \times d}$. 
\end{proof}

\begin{lemma}\label{lem:hilfslemma1}
    Let $(\Omega, \mathcal{F}, \mathbb{P})$ be a probability space, and let $V_i, W_i : \Omega \rightarrow \mathbb{R}^{d}$, $i \in \ldbrack 1,n \rdbrack$, be random vectors such that $\EWt{\Vert V_i \Vert^4}, \EWt{\Vert W_i \Vert^4} < \infty$. Then we have
    $$
        \mathbb{E}\bigg[\bigg\Vert n^{-1} \sum_{i=1}^n \left(V_iW_i^{\top} - \EW{V_iW_i^{\top}}\right)\bigg\Vert^2\bigg] \leq n^{-1} \sum_{i=1}^n \left(\EW{\left\Vert V_i \right\Vert^4}\right)^{1/2} \left(\EW{\left\Vert W_i \right\Vert^4}\right)^{1/2}.
    $$
\end{lemma}
\begin{proof}
    Jensen's inequality implies 
    \begin{align*}
        \mathbb{E}\bigg[\bigg\Vert n^{-1} \sum_{i=1}^n \left(V_iW_i^{\top} - \EW{V_iW_i^{\top}}\right)\bigg\Vert^2\bigg] &\leq n^{-1} \sum_{i=1}^n \EW{\left\Vert V_iW_i^{\top} - \EW{V_iW_i^{\top}}\right\Vert^2} \\
        &\leq n^{-1} \sum_{i=1}^{n} \mathbb{E}\left[\left\Vert V_i W_i^{\top} \right\Vert^2\right] \\
        &\leq n^{-1} \sum_{i=1}^n \left(\EW{\left \Vert V_i \right\Vert^4} \EW{\left\Vert W_i \right\Vert^4}\right)^{1/2},
    \end{align*}
    where the Cauchy--Schwarz inequality has been used in the last step. 
\end{proof}
\begin{cor}\label{cor:RVwithNormalMultiplikation}
    Let $(\Omega, \mathcal{F}, \mathbb{P})$ be a probability space, and let $V_i, W_i : \Omega \rightarrow \mathbb{R}^{d}$, $i \in \ldbrack 1,n \rdbrack$, be random vectors such that $\EWt{\Vert V_i \Vert^4} < \infty$ and $W_1,\ldots,W_n$ are independent $d$-dimensional Gaussian vectors with $W_i \sim \mathcal{N}(0,\Sigma_i)$ with $\Sigma_i \in \mathbb{R}^{d \times d}$. Then we have 
    \begin{align*}
        &\mathbb{E}\bigg[\bigg\Vert n^{-1} \sum_{i=1}^n \left(V_iW_i^{\top} - \EW{V_iW_i^{\top}}\right)\bigg\Vert^2\bigg] \leq n^{-1} \sum_{i=1}^n \left(\mathbb{E}\left[\left \Vert V_i \right\Vert^4\right]  \left(\big(\tr\left(\Sigma_i\right)\right)^2+ 2 \HS{\Sigma_i}^2\big)\right)^{1/2}.
    \end{align*}
\end{cor}
\begin{proof}
    As $X \sim \mathcal{N}(0,\Sigma)$ we have $\EWt{\Vert X \Vert^4} = (\tr(\Sigma))^2 + 2\HSt{\Sigma}^2$. This implies the claim by Lemma~\ref{lem:hilfslemma1}.
\end{proof}

\end{appendices}

\bibliographystyle{apalike2}
\bibliography{HF-RankTest}

\end{document}